\documentclass[11pt]{article}
\usepackage{amssymb,amsmath}
\usepackage{graphics}
\usepackage{float}
\usepackage{color,fullpage}
\usepackage{cite}
\usepackage{bm}
\usepackage[title,titletoc]{appendix}
\usepackage{algorithm,algorithmic}
\usepackage{multirow}
\usepackage{subcaption}
\usepackage{pgfplots}
\usepackage{amsthm}
\newcommand{\RR}{\mathbb{R}}
\newcommand{\NN}{\mathbb{N}}

\newcommand{\EE}{\mathbb{E}}
\newcommand{\dom}{{\mathrm{dom}}\,} 

\newcommand{\dist}{{\mathrm{dist}}}

\newcommand{\cD}{{\mathcal{D}}}

\DeclareMathOperator*{\Min}{minimize}

\newcommand{\st}{\mbox{subject to}}

\newtheorem{theorem}{Theorem}[section]
\newtheorem{lemma}{Lemma}[section]
\newtheorem{definition}{Definition}[section]

\makeatletter

\@addtoreset{equation}{section} \makeatother

\begin{document}

\title{Adaptively Sketched Bregman Projection Methods for Linear Systems}

\author{
Zi-Yang Yuan\thanks{Academy of Military Science of People's Liberation Army and Department of Mathematics, National University of Defense Technology.\texttt{yuanziyang11@nudt.edu.cn}}
\and Lu Zhang\thanks{Department of Mathematics, National University of Defense Technology,
Changsha, Hunan 410073, China.}
\and Hongxia Wang\thanks{
Department of Mathematics, National University of Defense Technology,
Changsha, Hunan, 410073, P.R.China. \texttt{wanghongxia@nudt.edu.cn}}
\and Hui Zhang\thanks{Corresponding author.
Department of Mathematics, National University of Defense Technology,
Changsha, Hunan 410073, China.  Email: \texttt{h.zhang1984@163.com}
}
}


\date{\today}

\maketitle

\begin{abstract}
The sketch-and-project, as a general archetypal algorithm for solving linear systems, unifies a variety of randomized iterative methods such as the randomized Kaczmarz and randomized coordinate descent. However, since it aims to find a least-norm solution from a linear system, the randomized sparse Kaczmarz can not be included. This motivates us to propose a more general framework, called sketched Bregman projection (SBP) method, in which we are able to find solutions with certain structures from linear systems. To generalize the concept of adaptive sampling to the SBP method, we show how the progress, measured by Bregman distance, of single step depends directly on a sketched loss function. Theoretically, we provide detailed global convergence results for the SBP method with different adaptive sampling rules. At last, for the (sparse) Kaczmarz methods, a group of numerical simulations are tested, with which we verify that the methods utilizing sampling Kaczmarz-Motzkin rule demands the fewest computational costs to achieve a given error bound comparing to the corresponding methods with other sampling rules.
\end{abstract}




\section{Introduction}
 The Kaczmarz method \cite{Kaczmarz1937Angenaherte} and its randomized variant\cite{2009AVershyin} for solving large-scale linear systems
 \begin{equation}\label{LS}
 Ax=b,
 \end{equation}
 where $A\in \RR^{m\times n}$ and $b\in \RR^m$
recently become very popular, mainly due to their cheap per iteration cost and low total computational complexity. Typically, the randomized Kaczmarz method is designed for solving highly over-determined linear systems; in other words, the number of samples $m$ is much larger than the dimension of variable $n$. In each iteration $k$, the current iterate $x^k$ is projected onto a hyperplane formed by randomly selecting a row $i_k$ of the linear systems \eqref{LS}, that is to obtain $x^{k+1}$ in the following way
  \begin{equation}\label{rk}
x^{k+1}=\arg\min_{x\in \RR^n} \|x-x^k\|^2, \quad\st\quad A_{i_k:}x=b_{i_k},
 \end{equation}
 where $A_{i_k:}$ is the $i_k$-row of $A$ randomly selected at iteration $k$ and $b_{i_k}$ is the $i_k$-entry of $b$. Recently, remarkable progress of the Kaczmarz method has been made; see for example \cite{2009AVershyin,2016Linear,7025269,7032145,Liu2015An,Petra2015Randomized,bai20181on,popa2012on,needell2014paved,leventhal2010randomized,haddock2019greed,bai2018on}.

 On the one hand, a sparse variant of the randomized Kaczmarz method was studied in \cite{2016Linear} to recover sparse solutions of possibly under-determined linear systems namely, where the number of samples $m$ is allowed to be smaller than the dimension $n$. In each iteration $k$ of the randomized sparse Kaczmarz method, the current iterate $x^k$ is projected via certain Bregman's distance onto the random hyperplane $\{x:A_{i_k:}x=b_{i_k}\}$. That is to obtain $x^{k+1}$ by solving the following linear constrained optimization problem
   \begin{equation}\label{srk}
x^{k+1}=\arg\min_{x\in \RR^n} D_f^{x^k_*}(x^k, x), \quad\st\quad A_{i_k:}x=b_{i_k},
 \end{equation}
 where $f$, called generating function, is some strongly convex function used to induce the Bregman distance $D_f^{x^k_*}(x^k, x)$.  Different from the classical random Bregman projection method \cite{1997Legendre}, the function $f$ in the randomized sparse Kaczmarz method could be nonsmooth.
 If we take $f(x)=\frac{1}{2}\|x\|^2$, then $D_f^{x^k_*}(x^k, x)=\frac{1}{2}\|x-x^k\|^2$ and the update in \eqref{srk} reduces to that in \eqref{rk}. In this sense, the scheme \eqref{srk} generalizes the randomized Kaczmarz method. A remarkable advantage behind is that the Bregman distance could characterize different geometries via choosing different generating functions $f$. For example, in the randomized sparse Kaczmarz method, the convex function
 $f(x)=\frac{1}{2}\|x\|^2+\lambda\|x\|_1$ is used to produce sparse solution $x^{k}$. See Figure \ref{geometric} as an example. Assume the generating function $f$ of the Bregman projection is $\frac{1}{2}\|x\|^2+\|x\|_1$. The pre-image of the Euclidean projection point $(1,0)$ is a line with measure $0$ in the $\text{x-y}$ plane. However, the Pre-image of Bregman projection point $(1,0)$ is the blue region. So comparing to the Euclidean projection, the Bregman projection is more capable to generate the sparse solution.

On the other hand, by replacing the random hyperplane with a random sketch of the original system and introducing the $B$-norm given by $\|\cdot\|_B:=\sqrt{\langle \cdot, B\cdot\rangle}$ where $B$ is a symmetric positive definite matrix, Gower and Richt$\acute{a}$rik in \cite{Gower2015} proposed the sketch-and-project framework
   \begin{equation}\label{sprk}
x^{k+1}=\arg\min_{x\in \RR^n} \|x^k- x\|_B^2, \quad\st\quad S^{\top}_{i_k}Ax=S^{\top}_{i_k}b,
 \end{equation}
where $S_{i_k}$ are random matrices drawn in an independent and identically distributed (i.i.d) fashion in each iteration, and matrix $B$ is a user-defined symmetric positive matrix used to define the geometry of the space. By varying these two parameters $S$ and $B$, the authors of \cite{Gower2015} can recover many existing variants of the Kaczmarz method as special cases, including the randomized Kaczmarz method, randomized Newton method, randomized coordinate descent method, random Gaussian pursuit, and variants of all these methods using blocks and important sampling. However, the $B$-norm is not general enough to capture \textsl{non-smooth} geometry like sparsity, and hence is impossible to recover sparse solutions. Such limitation could be overcome by utilizing the Bregman distance with a suitable generating function. This is the main motivation of our study in this paper.

Now, blending the Bregman projection scheme \eqref{srk} and the sketch-and-project method \eqref{sprk}, we propose a more general unified framework, called sketched Bregman projection (SBP) method, to solving linear systems. It updates $x^{k}$ via
\begin{equation}\label{urk}
x^{k+1}=\arg\min_{x\in \RR^n} D_f^{x^k_*}(x^k, x), \quad\st\quad S^{\top}_{i_k}Ax=S^{\top}_{i_k}b.
 \end{equation}
Specially noticing that when $f=\|x\|_{B}^2$ with $B$ being a symmetric positive definite matrix, \eqref{urk} reduces to \eqref{sprk}. At each iteration, when $i_k$ is sampled with a fixed probability, we will show that SBP method can have a linear convergence rate. At the same time, an adaptive sampling rule is also introduced. By tuning parameters of the adaptive sampling rule, many existing sampling rules such as max-distance, proportional to the sketch loss rule\cite{Gower2019}, capped sampling rule\cite{bai2018on} can be recovered. Moreover a new sampling rule called Sketch-Motzkin rule is proposed which can derive the sampling Kaczmarz-Motzkin rule in \cite{haddock2019greed}. Theoretical results demonstrate that the SBP methods equipped with these adaptive sampling rules above can have a faster convergence rate. Furthermore, we apply the theoretical results of SBP method to the Kaczmarz method and sparse Kaczmarz method to recover many existing results, recently obtained in \cite{2016Linear,2009AVershyin,haddock2019greed,bai2018on,Gower2019,Maxdis,Yuan2021}. Numerical simulations are tested to demonstrate that the sparse Kaczmarz method utilizing the sampling Kaczmarz-Motzkin rule has the least burden to achieve a given error bound.

In the next section, we recall some basic properties about the Bregman distances and projections. In Section 3, we formally propose the unified framework \eqref{urk}, and introduce the SBP method along with the sampling rules. In Section 4, convergence rates of SBP methods with non-adaptive sampling rule and adaptive sampling rules are analyzed. At the same time, by tuning the parameters the special cases of the adaptive sampling rules are also researched. In Section 5, we present a couple of examples of the proposed framework and analyze their convergence behaviour. In Section 6, we report some numerical results to compare the performance of different methods. In Section 7, we give a few concluding remarks and research directions for future work.
\begin{figure}
\centering
\includegraphics[width=0.5\textwidth]{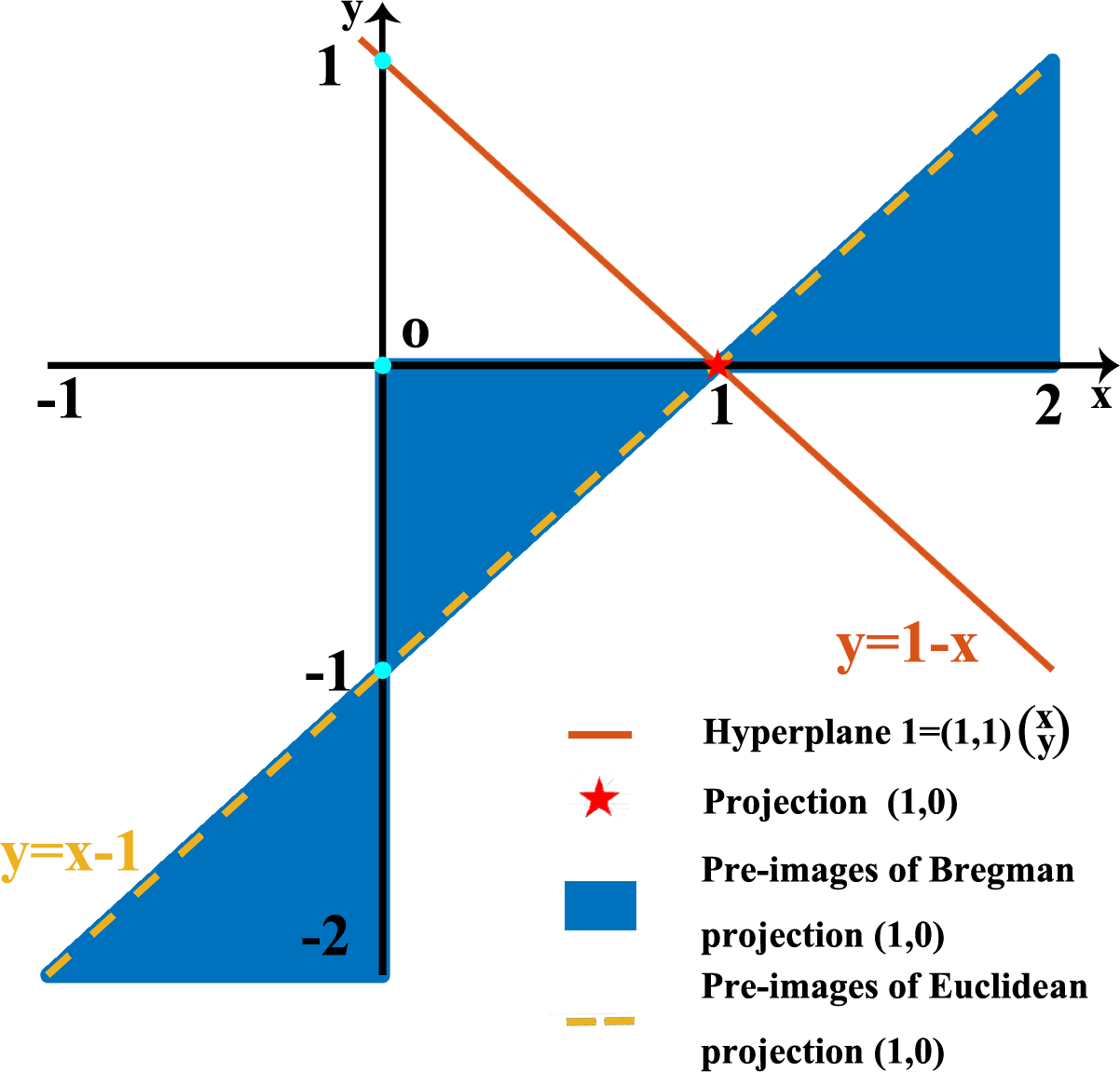}
\caption{Comparison between Bregman projection and Euclidean projection to generate the sparse point. The generating function $f$ of the Bregman projection is $\frac{1}{2}\|x\|^2+\|x\|_1$.}
\label{geometric}
\end{figure}

\section{Preliminaries}
Let $f:\RR^n\rightarrow (-\infty, +\infty]$ be a proper convex function. The effective domain of $f$ is given by $\dom f :=\{x\in \RR^n: f(x)< +\infty\}$.
The Fenchel conjugate of $f$ is defined as
$$f^*(y)=\sup_{x\in \RR^n}\{\langle x,y\rangle -f(x)\}.$$
The subdifferential of $f$  at $x$ is given by
$$\partial f (x):= \{ y \in \RR^n: f(u)\geq f(x)+ \langle y, u-x\rangle,\quad \forall u\in \RR^n \}.$$
We say that $f$ is subdifferentiable at $x\in \RR^n$ if $\partial f(x)\neq \emptyset$. The elements of $\partial f(x)$ are called the subgradients of $f$ at $x$.
We say that $f:\RR^n\rightarrow \RR$ is strongly convex with modulus $\mu>0$ if for any $x, y\in \RR^n$ and $x_*\in \partial f(x)$, we have
\begin{equation}
f(y)\geq f(x)+\langle x_*, y-x\rangle +\frac{\mu}{2}\|y-x\|^2.
\end{equation}

Throughout the paper, $\|\cdot\|$ is default to be the 2-norm, $\|\cdot\|_{\text{F}}$ is the Frobenius norm of a matrix. $(\cdot)^{\dagger}$ is the Moore -penrose pseudo inverse, and $\textbf{supp}(\cdot)$ means the set of the index whose value is not zero. For a real positive number $q$, $[q]:=\{1,2,\cdots,q\}$. $|\cdot|$ is the cardinality of a set.

\subsection{Bregman distance and projection}
We recall the definitions of the Bregman distance and Bregman projection, introduced in \cite{2016Linear}.
\begin{definition}[Bregman distance \cite{2016Linear}] Let $f:\RR^n\rightarrow \RR$ be a strongly convex function. The Bregman distance $D_f^{x_*}(x,y)$ between $x, y\in \RR^n$ with respect to $f$ and a subgradient $x_*\in\partial f(x)$ is defined by
$$D_f^{x_*}(x,y):=f(y)-f(x)-\langle x_*, y-x\rangle.$$
\end{definition}

\begin{definition}[Bregman projection \cite{2016Linear}] Let $f:\RR^n\rightarrow \RR$ be a strongly convex function and $C\subset \RR^n$ be a nonempty closed convex set. The Bregman projection of $x$ onto $C$ with respect to $f$ and $x_*\in\partial f(x)$ is the unique point $\Pi_C^{x_*}(x)\in C$ such that
$$\Pi_C^{x_*}(x)=\arg\min_{y\in C}D_f^{x_*}(x,y).$$
\end{definition}
It is not hard to see that the uniqueness of $\Pi_C^{x_*}(x)$ is implied by the strong convexity of $f$. The characterization below of the Bregman projection will play a vital role for deriving the linear convergence rate later on.

\begin{lemma}[\cite{2016Linear}]\label{lemma2.1}
Let $f:\RR^n\rightarrow \RR$ be a strongly convex function, and $C\subset \RR^n$ be a nonempty closed convex set. Then a point $\hat{x}\in C$ is the Bregman projection of $x$ onto $C$ with respect to $f$ and $x_*\in\partial f(x)$ if and only if there is some $\hat{x}_*\in \partial f(\hat{x})$ such that the following condition is  satisfied
\begin{equation}\label{ineq1}
D_f^{\hat{x}_*}(\hat{x},y)\leq D_f^{x_*}(x,y) - D_f^{x_*}(x,\hat{x}),\quad \forall y\in C.
\end{equation}
We call any such $\hat{x}_*$ an admissible subgradient for $\hat{x}=\Pi_C^{x_*}(x)$.
\end{lemma}
Notice that when $f=\frac{1}{2}\|x\|^2,$ Lemma \ref{lemma2.1} actually demonstrates the theoretical results in Pythagoras theorem namely for a right-angled triangle the square of the hypotenuse side is equal to the sum of squares of the other two sides.

\subsection{Sketching}
Define $\cD:=\{S_i\in \RR^{m\times \tau}, i=1,2,\cdots q\}$ to be the set of sketching matrices where $\tau\in\NN$ (not fixed) is the sketching size. In general, the size of the sketching matrices $\cD$ is allowed to be infinite, however, we assume that it is a finite set for simplicity of convergence analysis. In each iteration, we will choose a sketching matrix $S_i$ from $\cD$ with positive probability $p_i$. To this end, we let $\Delta_q^{\dag}$ denote the interior of the simplex in $\RR^q$, that is
$$\Delta_q^{\dag}:=\{p\in \RR^q: \sum_{i=1}^q p_i=1, p_i>0\}.$$
Let probabilities $p\in \Delta_q^{\dag}$ and $x$ be a random variable that takes values $x_i\in \{x_1, x_2,\cdots, x_q\}$ with probability $p_i$. Next we denote
$$\EE_p[x]:= \sum_{i=1}^q p_ix_i.$$
For simplicity, we introduce a random index $r$ that takes values $i\in \{1,2,\cdots, q\}$ with probability $p_i$, denoted by $r\sim p$. Then we also denote
$$\EE_{r\sim p}[x_r]=\EE_p[x] = \sum_{i=1}^q p_ix_i.$$
Using this type of language, we could define a random matrix $S$ as a sketch that takes values $S_i$ from $\cD$ with probability $p_i$.

\section{The proposed method}
In this section, we describe the proposed method \eqref{urk} in details. Recall that the sketched Bregman projection method is abbreviated by SBP.
\begin{table}
	\label{a1}
	\begin{tabular}{l}
		\hline
		Algorithm 1 Non-adaptive SBP method\\
		\hline 1: \textbf{Input}: $x^{0}=x^{0}_*=0 \in \mathbb{R}^{n}$, strongly convex function $f$, $A \in \mathbb{R}^{m \times n}, b \in \mathbb{R}^{m}$, $p \in \Delta_{q}^{\dagger}$,\\ $\quad \quad \quad\quad$ and $S=\left[S_{1}, \ldots, S_{q}\right]$, $S_i\in\mathbb{R}^{m\times\tau},i=1,\cdots,q$ \\
		2: for $k=0,1,2, \ldots$ do \\
		3: $\quad i_{k} \sim p$ \\
		4: $\quad y^k\in\arg\min_{y\in\mathbb{R}^{\tau}} f^*(x_*^k-A^{\top}S_{i_k}y)+\langle S_{i_k}^{\top}b, y\rangle$\\
		5: $\quad x^{k+1}_*=x_*^k-A^{\top}S_{i_k}y^k$\\
		6: $\quad  x^{k+1}=\nabla f^*(x_*^{k+1})$\\
		7: \textbf{Output}: last iterate $x^{k+1}$\\
		\hline	
	\end{tabular}
\end{table}
\subsection{The SBP method}
Let $f$ be a given strongly convex function and $r$ be a random mapping from $\NN$ onto $\{1, 2,\cdots q\}$ with its entry $r^k$ obeying probability $p^k\in \Delta^{\dag}_q$, i.e.,  $r^k$ is a random variable taking values from $\{1, 2,\cdots q\}$  with probability $p^k_i$. Given $x^k$ and $x_*^k\in\partial f(x^k)$, the update of $x^k$ in the $(k+1)$th step of SBP method reads as
\begin{equation}\label{urk1}
x^{k+1}=\arg\min_{x\in \RR^n} D_f^{x_*^k}(x^k, x), \quad\st\quad S_{i_k}^{\top}Ax=S_{i_k}^{\top}b.
 \end{equation}
To deduce an admissible subgradient $x_*^{k+1}$ for $x^{k+1}$, we reformulate the optimization problem in \eqref{urk1} as follows
\begin{equation}\label{urk1.1}
\Min_{x\in \RR^n} F(x)+G(M_k x),
 \end{equation}
 where $F(x):=f(x)-\langle x_*^k, x\rangle$, $G(z):=\delta_{\{S_{i_k}^{\top}b\}}(z)$, and $M_k:=S_{i_k}^{\top}A$. Its Fenchel-Rockafellar dual problem \cite{Bauschke2017} is
 \begin{equation}\label{urk1.2}
\Min_{y\in\mathbb{R}^{\tau}} D(y):= F^*(-M_k^{\top} y)+G^*(y)=f^*(x_*^k-A^{\top}S_{i_k}y)+\langle S_{i_k}^{\top}b, y\rangle.
 \end{equation}
Let $y^k$ be a solution to \eqref{urk1.2}, which must exist but not necessarily be unique. Then, following from the KKT  \cite{Bauschke2017} condition of \eqref{urk1.1}-\eqref{urk1.2}, that is
 \begin{eqnarray}\label{app3.2}
\left\{\begin{array}{lll}
- A^{\top}S_{i_k}y \in   \partial f(x)-x_*^k, \\[5pt]
S_{i_k}^{\top}Ax = S_{i_k}^{\top}b,
\end{array} \right.
\end{eqnarray}
we have that
\begin{equation}\label{xupdate}
 x^{k+1}=\nabla f^*(x_*^k-A^{\top}S_{i_k}y^k).
\end{equation}
Now, denote
$$x_*^{k+1}:=x_*^k-A^{\top}S_{i_k}y^k.$$
From \eqref{xupdate} we immediately have that $x_*^{k+1}\in \partial f(x^{k+1})$. In other words, such $x_*^{k+1}$ is an admissible subgradient for $x^{k+1}$.
Let $\overline{x}$ be the ground truth. Next, we will give an error bound of the iterates, which helps us bound $D_f^{x_*}(x,\overline{x})$ by $\|Ax-b\|$.
\begin{lemma}[Theorem 4.12 in \cite{2016Linear}]\label{2016linear}
Consider the linearly constrained optimization problem \eqref{urk1} with $A \in$ $\mathbb{R}^{m \times n}, b \in \mathcal{R}(A),$ and strongly convex $f: \mathbb{R}^{n} \rightarrow \mathbb{R} .$ Let $x^{0} \in \mathbb{R}^{n}$ and $x^{0}_{*} \in \partial f\left(x^{0}\right) \cap$
	$\mathcal{R}\left(A^{\top}\right)$ be given.  If the subdifferential mapping of $f$ is calm at the unique solution $\overline{x}$ of (15) and the collection $\left\{\partial f(\overline{x}), \mathcal{R}\left(A^{\top}\right)\right\}$ is linearly regular, then there exists $\gamma>0$ such that for all $x \in \mathbb{R}^{n}$ and $x_{*} \in \partial f(x) \cap \mathcal{R}\left(A^{\top}\right)$ with $D_{f}^{x_{*}}(x, \overline{x}) \leq D_{f}^{x^{0}_{*}}\left(x^{0}, \overline{x}\right)$, we have
\begin{eqnarray}\label{EB1}
	\gamma \cdot D_{f}^{x_{*}}(x, \overline{x}) \leq \|A x-b\|^{2}.
	\end{eqnarray}
\end{lemma}

We here briefly recall the definitions of the calmness and linear regularity; for more details please refer to Section 4 in \cite{2016Linear}. A set-valued mapping $S: \mathbb{R}^{n} \rightarrow \mathbb{R}^{m}$ is said to be calm at $\hat{x} \in \mathbb{R}^{n}$ if $S(\hat{x}) \neq \emptyset$ and for any $x\in B_{\varepsilon}(\hat{x})$, there are some constants $\epsilon$ and $L>0$ such that
$$ S(x) \subset S(\hat{x})+L \cdot\|x-\hat{x}\| \cdot B_{1}(0),$$
where $B_\varepsilon(\hat{x})=\{x|\|x-\hat{x}\|\leq\varepsilon\}.$  Two closed convex sets $C_1, C_2$ with nonempty intersection $C=C_1\bigcap C_2$ are said to be linearly regular if there exists $\beta>0$ such that for all $x\in \mathbb{R}^n$ we have
$$\dist(x,C)\leq \beta (\dist(x,C_1)+\dist(x,C_2)).$$

\noindent It has been well verified in \cite{2016Linear} that the differential mapping or sub-differential mapping of $f(x)=\frac{1}{2}\|\mathbf{x}\|^2$ and $f(x)=\frac{1}{2}\|\mathbf{x}\|^2+\lambda\|x\|_1$ are calm at each $x\in\mathbb{R}^n.$
Besides, from Lemma \ref{2016linear}, we can conclude that for any $k$ except $x^k=\overline{x}$, $(x^k-\overline{x})\notin\mathrm{Null}(A)$, if $D_{f}^{x^k_{*}}(x^k, \overline{x}) \leq D_{f}^{x^{0}_{*}}\left(x^{0}, \overline{x}\right)$ .

\subsection{The sampling rule}
The sampling rule of $i_k$ in the SBP method can be divided into non-adaptive and adaptive types. Non-adaptive sampling rules use some fixed probabilities $p\in\Delta_{q}$. Well-known examples include the uniform sampling rule and the rule in \cite{2009AVershyin} which selects the rows with the probability proportional to $\|A_{i_k:}\|^2$. The details of the SBP method with non-adaptive sampling rules are provided in Algorithm 1.

Although the non-adaptive SBP method can be implemented with low computational cost and can be analyzed in a relatively simple way, one of its remarkable drawbacks is that we may sample the same index consecutively; say for example $i_k=i_{k-1}$. Then,
\begin{eqnarray}
\begin{split}
&x^{k+1}=x^{k} \in\arg\min D_f^{x^k_*}(x^k,x)\\
& \st~S_{i_{k-1}}^{\top}Ax =S_{i_{k-1}}^{\top}b
\end{split}.
\end{eqnarray}
As a result, no progress will be made in the updated step. This motivates us to design adaptive sampling rules to speed up convergence. Intuitively, we should assign non-zero probabilities to the indices that have non-zero sketched loss values, defined by
\begin{eqnarray}
 g_{i_k}\left(x^{k}\right)=\left\|A x^{k}-b \right\|_{S_{i_k}(S_{i_k}^{\top}AA^{\top}S_{i_k})^{\dagger}S_{i_k}^{\top}}^2. \label{exactloss}
  \end{eqnarray}
 Consequently, the same index will never be consecutively chosen because the sketched loss value will disappear, as showed in the lemma below.

\begin{lemma}\label{S0}
 	Denote $H_{i_k}:=S_{i_k}(S_{i_k}^{\top}AA^{\top}S_{i_k})^{\dagger}S_{i_k}^{\top}$ and consider the sketched loss $g_{i_k}\left(x^{k}\right)=\left\|A x^{k}-b \right\|_{H_{i_k}}^{2}$ with $x^k$ generated by the SBP \eqref{urk}. We have that
$$
g_{i_{k}}\left(x^{k+1}\right)=0, \quad \forall k \geq 0.
$$
\end{lemma}
\begin{proof}
Let $Z_{i_k}:=A^{\top}H_{i_k}A$. By the definition of $g_{i_k}(x^k)$, we deduce
$$
g_{i_{k}}\left(x^{k+1}\right)=\left\|x^{k+1}-\overline{x}\right\|_{Z_{i_{k}}}^{2}=\left\langle \left(x^{k+1}-\overline{x}\right),Z_{i_{k}} \left(x^{k+1}-\overline{x}\right)\right\rangle.
$$

Therefore, the conclusion follows by observing that
\begin{eqnarray*}
Z_{i_{k}}\left(x^{k+1}-\overline{x}\right) &=& A^{\top}S_{i_k}(S^{\top}_{i_k}AA^{\top}S_{i_k})^{\dag}S_{i_k}^{\top}A\left(x^{k+1}-\overline{x}\right) \\
&=&A^{\top}S_{i_k}(S_{i_k}^{\top}AA^{\top}S_{i_k})^{\dag}(S_{i_k}^{\top}Ax^{k+1}-S_{i_k}^{\top}b)=0.
\end{eqnarray*}
\end{proof}

\noindent With this key result at hand, we are able to design adaptive sampling rules.
\begin{table}
	\label{a2}
	\begin{tabular}{l}
		\hline
		Algorithm 2 Sampling adaptive SBP method\\
		\hline ~1: \textbf{Input}: $x^{0}=x^{0}_*=0 \in \mathbb{R}^{n}$, $\theta \in [0,1]$, strongly convex function $f$, $A \in \mathbb{R}^{m \times n}, b \in \mathbb{R}^{m},$ \\
		$\quad \quad \quad\quad$ $S=\left[S_{1}, \ldots, S_{q}\right],$ $S_i\in\mathbb{R}^{m\times\tau},i=1,\cdots,q$ \\
		~2: for $k=0,1,2, \ldots$ do \\
		~3: $\quad\tau_k\in\binom{[q]}{\beta_k}\sim p_1^k$ $,p_1^k \in \Delta_{|\binom{[q]}{\beta_k}|}^{\dagger}$, and $\beta_k\leq q$ \\
		~4: $\quad g_{i}\left(x^{k}\right)=\left\|A x^{k}-b\right\|_{H_{i}}$ for $i\in\tau_k$ where $H_{i}:=S_i(S_{i}^{\top}AA^{\top}S_{i})^{\dagger}S_{i}^{\top}$ \\
		~5: $\quad\mathcal{W}_{k}=\left\{i \mid g_{i}\left(x^{k}\right) \geq \theta \max _{j\in\tau_k} g_{j}\left(x^{k}\right)+(1-\theta) \mathbb{E}_{j \sim p_2^k}\left[g_{j}\left(x^{k}\right)\right]\right\}$,\\
		~~$\quad$ where $p^{k}_2 \in \Delta_{q}^{\dagger}$ such that $\textbf{supp}\left(p^{k}_2\right) \subset \tau_{k}$\\
		~6: $\quad i_k\sim p^k_3$, where $p^{k}_3 \in  \Delta_{q}^{\dagger}$ so that $\textbf{supp}\left(p^{k}_3\right) \subset \mathcal{W}_{k}$\\
		~7: $\quad y^k\in\arg\min_{y\in\mathbb{R}^{\tau}} f^*(x_*^k-A^{\top}S_{i_k}y)+\langle S_{i_k}^{\top}b, y\rangle$\\
		~8: $\quad x^{k+1}_*=x_*^k-A^{\top}S_{i_k}y^k$\\
		~9: $\quad  x^{k+1}=\nabla f^*(x_*^{k+1})$\\
		10: \textbf{Output}: last iterate $x^{k+1}$\\
		\hline	
	\end{tabular}
\end{table}
At each iteration $k$, we consider the class of subsets $\binom{[q]}{\beta_k}$, each of which is a subset of $[q]$ containing $\beta_k$ number chosen from the set $[q]$. Assume that each subset will be sampled with probabilities $p_1^k\in\Delta_{|\binom{[q]}{\beta_k}|}^{\dagger}$. For example, one may consider the following weighted probabilities or uniform probability. $$p_1^k(\tau_k)=\frac{\sum_{i\in\tau_k}\| S_i^{\top}A\|_{\textrm{F}}^2}{\sum_{\tau_k\in\binom{[q]}{\beta_k}}\sum_{i\in\tau_k}\| S_i^{\top}A\|_{\textrm{F}}^2}.$$ The reason to sample the index from one of the $[\binom{[q]}{\beta_k}]$ blocks is to add randomness to alleviate the algorithm from early stagnation.

Next, we construct an index set $\mathcal{W}_{k}$, which contains indices whose sketched losses are large than the weighted sum of the maximal sketched loss and the expectation $\mathbb{E}_{i \sim p_2^k}\left[g_{i}\left(x^{k}\right)\right]$. More concretely, the index set can be defined by
$$\mathcal{W}_{k}:=\left\{i\in\tau_k\mid g_{i}\left(x^{k}\right) \geq \theta \max _{j\in\tau_k} g_{j}\left(x^{k}\right)+(1-\theta) \mathbb{E}_{j \sim p_2^k}\left[g_{j}\left(x^{k}\right)\right]\right\},$$
where $p^{k}_2 \in \Delta_{q}^{\dagger}$ such that $\textbf{supp}\left(p^{k}_2\right) \subset\tau_{k}$ and the input parameter $\theta \in[0,1]$ controls how aggressive the sampling method is.
Finally, the probability $p_{3}^{k}$ only considers those indices that are contained in the set $\mathcal{W}_{k}$. The details can be found in Algorithm 2.

In the remaining,  from both theoretical and numerical aspects, we will demonstrate that the sampling adaptive SBP method with finely tuning parameters will converge faster.


\section{Convergence analysis}
In this section, we deduce convergence results for the SBP method with non-adaptive and adaptive sampling rule, respectively. To this end, we first introduce some important spectral constants for formulating the convergence rates of the SBP method.
\subsection{ Important spectral constants}
First, we require the exactness assumption, introduced in \cite{richtarik2020stochastic}. Recall that $Z_{i_k}:=A^{\top}H_{i_k}A$. The lemma below shows that $Z_{i}$ is actually an orthogonal projection matrix.
\begin{lemma}\label{orth}
	Let $Z_{i_k}:=A^{\top}H_{i_k}A$. Then we have
	\begin{eqnarray}\label{pequal}
		Z_iZ_i=Z_i~\text{and}~(I-Z_i)Z_i=0.
	\end{eqnarray}
\end{lemma}
\begin{proof} It follows from that
	\begin{eqnarray*}
		Z_iZ_i&=&A^{\top}S_i(S_i^{\top}AA^{\top}S_i)S_i^{\top}AA^{\top}S_i(S_i^{\top}AA^{\top}S_i)^{\dag}S_i^{\top}A\\
		&=&A^{\top}S_i(S_i^{\top}AA^{\top}S_i)^{\dag}S_i^{\top}A=Z_i.
	\end{eqnarray*}
\end{proof}
\noindent We say that the \textsl{exactness assumption} \cite{richtarik2020stochastic} holds for $(p,\cD)$ if
\begin{equation}\label{EA}
\textrm{Null}(A)=\textrm{Null} (\EE_{i\sim p}[Z_i]).
\end{equation}
It has been verified that the exactness assumption holds trivially for most sketching techniques \cite{richtarik2020stochastic,Gower2019}. Based on the exactness assumption, one can conclude that the expected sketched loss $\mathbb{E}_{i\sim p}\left[g_i(x)\right]=\mathbb{E}_{i\sim p}\left[\|Ax-b\|_{H_{i}}^2\right]=0$ if and only if $Ax=b$; the argument can be found for example from \cite{Gower2019}. Next, we introduce some important spectral constants.


\begin{definition}
	Let $p \in\Delta_q^{\dagger}$. Denote
	\begin{eqnarray}\label{spec1}
	\sigma_{p}^{2}(S) := \min _{v \notin \mathrm{Null(A)}} \frac{\left\|v\right\|_{\mathbb{E}_{i \sim p}\left[Z_{i}\right]}^{2}}{\|v\|^{2}},
	\end{eqnarray}
	and
	\begin{eqnarray}\label{spec2}
	\sigma_{\infty}^{2}(S) := \min _{v \notin \mathrm{Null(A)}} \max _{i=1, \ldots, q} \frac{\left\| v\right\|_{Z_{i}}^{2}}{\|v\|^{2}} .
	\end{eqnarray}
	\end{definition}

\begin{lemma}\label{spectrum1}
Assume the conditions in Lemma \ref{2016linear} hold. Let $p \in \Delta_{q}^{\dagger}$ and  the iterates $x^{k}$ be generated by Algorithm 1 satisfying $D_{f}^{x^k_{*}}(x^k, \overline{x}) \leq D_{f}^{x^{0}_{*}}\left(x^{0}, \overline{x}\right)$. Then, we have
\begin{equation}\label{spectrum2}
\begin{aligned}
\max _{i=1, \ldots, q} g_{i}\left(x^{k}\right) & \geq \sigma_{\infty}^{2}(S)\left\|x^{k}-\overline{x}\right\|^{2}, \\[5pt]
\mathbb{E}_{i \sim p}\left[g_{i}\left(x^{k}\right)\right] & \geq \sigma_{p}^{2}(S)\left\|x^{k}-\overline{x}\right\|^{2}.
\end{aligned}
\end{equation}
\end{lemma}
\begin{proof}
For $\max _{i=1, \ldots, q} g_{i}\left(x^{k}\right) $, we deduce
\begin{equation}
\begin{aligned}
\frac{\max _{i=1, \ldots, q} g_{i}\left(x^{k}\right)}{\left\|x^{k}-\overline{x}\right\|^2} {=} \max _{i=1, \ldots, q} \frac{\left\|x^{k}-\overline{x}\right\|_{Z_{i}}^{2}}{\left\|x^{k}-\overline{x}\right\|^{2}}\geq \min _{v \notin \mathrm{Null(A)}} \max _{i=1, \ldots, q} \frac{\left\| v\right\|^2_{Z_{i}}}{\|v\|^{2}} {=} \sigma_{\infty}^{2}(S).
\end{aligned}
\end{equation}
Analogously,
\begin{equation}
\begin{aligned}
\frac{\mathbb{E}_{i \sim p}\left[g_{i}\left(x^{k}\right)\right]}{\left\|x^{k}-\overline{x}\right\|^{2}} = \frac{\mathbb{E}_{i \sim p}\left[\left\|x^{k}-\overline{x}\right\|_{Z_{i}}^{2}\right]}{\left\|x^{k}-\overline{x}\right\|^{2}} \geq \min _{v \notin \mathrm{Null(A)}} \frac{\mathbb{E}_{i \sim p}\left[\left\| v\right\|_{Z_{i}}^{2}\right]}{\|v\|^{2}}= \sigma_{p}^{2}(S) .
\end{aligned}
\end{equation}
\end{proof}
		\begin{definition}
			Let $p_1^k\in\Delta_{|\binom{[q]}{\beta_k}|}^{\dagger}$, $p_2^{k}\in\Delta_{\beta_k}^{\dagger}$, and $\tau_k\in\binom{[q]}{\beta_k}$. Denote
			\begin{eqnarray}\label{spectrum3}
			\sigma_{p_1^k, p^{k}_2}(\beta_k,S) := \min _{v \notin \mathrm{Null(A)}} \frac{\left\|v\right\|_{\mathbb{E}_{\tau_k\sim p_1^k}\left[\mathbb{E}_{i \sim p^{k}_2}\left[Z_i\right]\right]}^{2}}{\|v\|^{2}},
			\end{eqnarray}
			and
			\begin{eqnarray}\label{spectrum4}
			\sigma_{p_1^k,\infty}(\beta_k,S):= \min _{v \notin \mathrm{Null(A)}}\mathbb{E}_{\tau_k\sim p_1^k}\left[ \max _{i\in\tau_k} \frac{\left\| v\right\|_{Z_i}^{2}}{\|v\|^{2}}\right] .
			\end{eqnarray}
		\end{definition}
\noindent
Similar to Lemma \ref{spectrum1}, we can get the results below.
\begin{lemma}\label{spectrum11}
Assume the conditions in Lemma \ref{2016linear} hold. Let $p_1^k\in\Delta_{|\binom{[q]}{\beta_k}|}^{\dagger}$, $p_2^{k}\in\Delta_{\beta_k}^{\dagger}$, $\tau_k\in\binom{[q]}{\beta_k}$ and the iterates $x^{k}$ be generated by Algorithm 2 satisfying $D_{f}^{x^k_{*}}(x^k, \overline{x}) \leq D_{f}^{x^{0}_{*}}\left(x^{0}, \overline{x}\right)$. Then, we have
	\begin{equation}
	\begin{aligned}
\mathbb{E}_{\tau_k\sim p_1^k}\left[	\max _{i\in \tau_k} g_{i}\left(x^{k}\right)\right] & \geq \sigma_{p_1^k,\infty}^{2}(\beta_k, S)\left\|x^{k}-\overline{x}\right\|^{2}, \\[5pt]
	\mathbb{E}_{\tau_k\sim p_1^k}\left[\mathbb{E}_{i \sim p^{k}_2}\left[g_{i}\left(x^{k}\right)\right]\right] & \geq \sigma_{p^k_1, p^{k}_2}^{2}(\beta_k,S)\left\|x^{k}-\overline{x}\right\|^{2}.
	\end{aligned}
	\end{equation}
		\end{lemma}
\noindent Next, we show that the four spectral constants are always less than one. Moreover, if the exactness assumption and the assumptions in Lemma \ref{2016linear} hold, then the spectral constants are strictly greater than zero.
\begin{lemma}\label{relationshipspectral}
Assume the conditions in Lemma \ref{2016linear} hold, and let $p \in \Delta_{q}^{\dagger}$,  $p_1^k\in\Delta_{|\binom{[q]}{\beta_k}|}^{\dagger}$, $p^{k}_{2}\in\Delta_{\beta_k}^{\dagger}$, and the set of sketching matrices $\left\{S_{1}, \ldots, S_{q}\right\}$ be such that exactness assumption holds. Then, we have
	\begin{itemize}
		\item[1.] 	$
		0<\sigma_{p}^{2}(S)\leq  \sigma_{\infty}^{2}(S) \leq 1
		.$
		\item[2.]	$
	 \sigma_{p_1^k,\infty}^2(\beta_k,S)\leq \sigma_{\infty}^{2}(S)
		.$
				\item [3.]
		$
		0<\sigma_{p_1^k,p^{k}_2}^2(\beta_k,S) \leq \sigma_{p_1^k,\infty}^2(\beta_k,S) \leq 1
		.$
	\end{itemize}
	\end{lemma}
\begin{proof}
	Let us show the relationship 1 firstly. By \eqref{spec1}, we deduce
	$$
	\begin{aligned}
	\sigma_{p}^{2}(S) & {=} \min _{v \notin \operatorname{Null}\left( A\right)} \frac{\left\| v\right\|_{\mathbb{E}_{i \sim p}\left[Z_i\right]}^{2}}{\|v\|_{}^{2}} \\
	& \stackrel{(\ref{EA})}{=} \min _{ v \notin \operatorname{Null}\left(\mathbb{E}_{i \sim p}\left[Z_i\right]\right)} \frac{\left\| v\right\|_{\mathbb{E}_{i \sim p}\left[Z_i\right]}^{2}}{\|v\|^2}=\sigma_{\min }^{+}\left(\mathbb{E}_{i \sim p}\left[Z_i\right]\right)>0 .
	\end{aligned}
	$$
On the other hand, we have
	$$
	\begin{aligned}
	\sigma_{p}^{2}(S) & {=} \min _{v \notin \operatorname{Null}\left(A\right)} \frac{\left\|v\right\|_{\mathbb{E}_{i \sim p}\left[Z_i\right]}^{2}}{\|v\|^2} \\
	& {=} \min _{v \notin \operatorname{Null}\left(A\right)} \frac{\mathbb{E}_{i \sim p}\left[\left\| v\right\|_{Z_i}^{2}\right]}{\|v\|^2} \\
	&\leq \min _{v \notin \operatorname{Null}\left( A\right)} \max _{i=1, \ldots, q} \frac{\left\| v\right\|_{Z_i}^{2}}{\|v\|^{2}}=\sigma_{\infty}^{2}(S) .
	\end{aligned}
	$$
	Last, by using the fact in Lemma \ref{orth} that the symmetric matrix $Z_i$ is an orthogonal projection, we get
	$$
	\sigma_{\infty}^{2}(S) = \min _{v \notin \mathrm{Null(A)}} \max _{i=1, \ldots, q} \frac{\left\| v\right\|_{Z_i}^{2}}{\|v\|^{2}} {\overset{\eqref{pequal}}{=}}\min _{v \notin \mathrm{Null(A)}} \max _{i=1, \ldots, q} \frac{\left\|Z_i v\right\|^{2}}{\left\| v\right\|^{2}} \leq \max _{i=1, \ldots, q} \frac{\left\| v\right\|^{2}}{\left\| v\right\|^{2}}=1 .
	$$
The relationship 2 follows from
	\begin{eqnarray*}
		\sigma^2_{p_1^k,\infty}(\beta_k,S)&= &\min _{v \notin \mathrm{Null(A)}}\mathbb{E}_{\tau_k\sim p_1^k}\left[ \max _{i\in\tau_k} \frac{\left\| v\right\|_{Z_i}^{2}}{\|v\|^{2}}\right]\\
		&\leq&  \min _{v \notin \mathrm{Null(A)}} \max _{i=1, \ldots, q} \frac{\left\| v\right\|_{Z_i}^{2}}{\|v\|^{2}}\\
		&=&\sigma_{\infty}^2(S).
	\end{eqnarray*}

\noindent It remains to show the relationship 3. For $\mathbb{E}_{\tau_k\sim p_1^k}\left[\mathbb{E}_{i \sim p^{k}_2}\left[Z_i\right]\right]$, using the concept of joint probability, we know there is $p_3^k\in\Delta_q^{\dagger}$ so that $\mathbb{E}_{\tau_k\sim p_1^k}\left[\mathbb{E}_{i \sim p^{k}_2}\left[Z_i\right]\right]=\mathbb{E}_{i \sim p_3^k}\left[Z_i\right].$ Thus, by repeating the argument for the relationship 1, we can show that $\sigma^2_{p_1^k,p_2^k}(\beta_k,S)> 0$ and $\sigma^2_{p_1^k,p_2^k}(\beta_k,S)\leq\sigma^2_{p_1^k,\infty}(S).$ This completes the proof.
\end{proof}

\subsection{Convergence of non-adaptive SBP}
Now, we are ready to present the linear convergence of the non-adaptive SBP method.
\begin{theorem}\label{main1}
	Let the probabilities $p\in\Delta_q^{\dagger}$ be given.  If $f$ is a $\mu$-strongly convex function such that the conditions in Lemma \ref{2016linear} holds, and the initialization $x^{0} \in \mathbb{R}^{n}$ and $x^{0}_{*} \in \partial f\left(x^{0}\right) \cap$
	$\mathcal{R}\left(A^{\top}\right)$, then the iterates in Algorithm 1 converge linearly in the sense that
	$$\mathbb{E}\left[D_{f}^{x^{k+1}_*}\left(x^{k+1}, \bar{x}\right)\right] \leqslant(1-\frac{\mu\gamma\sigma_p^2(S)}{2\|A\|^2}  )^{k+1} D_f^{x^0_*}(x^0,\overline{x}).$$
\end{theorem}
\noindent
To show Theorem \ref{main1}, we first deduce the following sufficient descent property.
\begin{lemma}\label{NT1}
	Denote $H_i:=S_i(S_i^{\top}AA^{\top}S_i)^{\dagger}S_i^{\top}$, $g_i(x):=\|Ax-b\|^2_{H_i}$, and $p\in\Delta^{\dagger}_q$. Then we have,
	\begin{eqnarray}\label{inequal1}
		\mathbb{E}_{i_k\sim p}[D_f^{x^{k+1}_*}(x^{k+1},\overline{x})|x^k]\leq D_f^{x^k_*}(x^k,\overline{x})-\frac{\mu}{2}\mathbb{E}_{i_k\sim p}[g_{i_k}(x^k)].
	\end{eqnarray}
\end{lemma}
\begin{proof}
	Denote $\mathcal{L}_k:=\{x:S^{\top}_{i_k}Ax=S^{\top}_{i_k}b\}$ and the orthogonal projection of $x^k$ onto $\mathcal{L}_k$ by $\tilde{x}^k$. Using the strong convexity of $f$ and the fact that $x^{k+1}\in\mathcal{L}_k$, we have,
	\begin{eqnarray}\label{t1}
	D_f^{x^k_*}(x^k,x^{k+1})\geq\frac{\mu}{2}\|x^{k}-x^{k+1}\|^2\geq\frac{\mu}{2}\|x^k-\tilde{x}^{k}\|.
	\end{eqnarray}
	Note that $\tilde{x}^k$ is the orthogonal projection of $x^k$ onto $\mathcal{L}_k$, then we have
	$$\tilde{x}^k\in\arg\min_{x\in \RR^n} \|x-x^k\|^2, \quad\st\quad x\in\mathcal{L}_k.$$
Equivalently,	
	\begin{eqnarray}
	\tilde{x}^k=x^k-A^{\top}S_{i_k}(S_{i_k}AA^{\top}S_{i_k})^{\dagger}S_{i_k}^{\top}(Ax^k-b).
	\end{eqnarray}
	Thereby, we derive that
	\begin{eqnarray}
	\|x^{k}-\tilde{x}^{k}\|^2&=&\left(A x^{k}-b\right)^{\top} S_{i_k}\left(S_{i_k}^{\top} A A^{\top} S_{i_k}\right)^{\dagger} S_{i_k}^{\top}\left(A x^{k}-b\right)\nonumber\\
	&=&\left(A x^{k}-b\right)^{\top} H_{i_k}\left(A x^{k}-b\right)\nonumber\\
	&=&\left\|A x^{k}-b \right\|_{H_{i_k}}^{2}.
	\label{t2}
	\end{eqnarray}
\noindent
From (\ref{t1}), (\ref{t2}) and the expression of $g_i(x)$, we have
\begin{eqnarray}\label{ft1}
D^{x_*^{k}}_f\left(x^{k}, x^{k+1}\right) \geqslant\frac{\mu}{2} g_{i_k}\left(x^{k}\right).
\end{eqnarray}
 Involving (\ref{ineq1}) in Lemma \ref{lemma2.1} and using \eqref{ft1}, we get
	\begin{eqnarray}\label{inequall}
		D_f^{x^{k+1}_*}(x^{k+1},\overline{x})\leq D_f^{x^k_*}(x^k,\overline{x})-\frac{\mu}{2}g_{i_k}(x^k).
	\end{eqnarray}
Finally, taking expectation conditioned on $x^k$, we get the desired result.
\end{proof}

Based on \eqref{inequall}, we can also know that $D_{f}^{x^k_{*}}(x^k, \overline{x}) \leq D_{f}^{x^{0}_{*}}\left(x^{0}, \overline{x}\right)$, then if the conditions in Lemma \ref{2016linear} are satisfied and , $x^0_*\in\partial f(x^0)\cap\mathcal{R}(A^{\top}),$ we can conclude that the iterates $\{x^k\}$ generated by the SBP method keep $(x^k-\overline{x})\notin\textrm{Null}(A),$ except $x^k=\overline{x}$. Next ,we can show that the term $\mathbb{E}_{i_k\sim p}[g_{i_k}(x^k)]$ can be controlled by the residual $A x^{k}-b$, as shown below.
\begin{lemma}\label{NT2} If the conditions in Theorem \ref{main1} are held, then the iterates in Algorithm 1 satisfy
	\begin{eqnarray}\label{Le5}
\mathbb{E}_{i_k\sim p}[g_{i_k}(x^k)]\geq\frac{\sigma^2_p(S)}{\|A\|^2}  \cdot\left\|A x^{k}-b\right\|^{2}.
	\end{eqnarray}
\end{lemma}
\begin{proof}
Recall \eqref{spectrum2}, we have that
\begin{eqnarray*}
	\mathbb{E}_{i_k\sim p}[g_{i_k}(x^k)]&\geq&\sigma_p^2(S)\left\|x^{k}-\bar{x}\right\|^{2}\\
	&\geq&\frac{\sigma^2_p(S)}{\|A\|^2}  \cdot\left\|Ax^{k}-b\right\|^{2},\label{ps}
\end{eqnarray*}
where the last inequality follows from $\left\|A x^{k}-b\right\|=\left\|A x^{k}-A \bar{x}\right\| \leqslant\|A\| \cdot\left\|x^{k}-\bar{x}\right\|.$
\end{proof}
\noindent Combing the Lemma \ref{NT1} and Lemma \ref{NT2} above, we now show Theorem \ref{T11}. Theorem \ref{main1} will follows naturally.
\begin{theorem}\label{T11}
	Suppose that $p\in \Delta_{q}^{\dagger}$. If the conditions in Theorem \ref{main1} are held,
	then the non-adaptive SBP converges linearly in the sense that,
	$$\mathbb{E}\left[D_{f}^{x^{k+1}_*}\left(x^{k+1}, \bar{x}\right)\right] \leqslant(1-\frac{\mu\cdot\gamma\cdot\sigma_p^2(S)}{2\|A\|^2}  )\mathbb{E}[D_f^{x^k_*}(x^k,\overline{x})].$$
\end{theorem}
\begin{proof}
	Combing (\ref{EB1}) and (\ref{Le5}), we can obtain that
	\begin{eqnarray}\label{inequal2}
	\mathbb{E}_{i_k\sim p}\left[g_{i_k}\left(x^{k}\right)\right] \geq \frac{\gamma \cdot\sigma^2_p(S)}{\| A \|^{2}} D_{f}^{x^{k}_*}\left(x^{k}, \bar{x}\right).
	\end{eqnarray}
By combing with \eqref{inequal1} and \eqref{inequal2}, we finish the proof.
\end{proof}
\begin{table}
	\centering
	\caption{Special cases of adaptive SBP method by tuning different parameters.}
	\label{special cases}
	\resizebox{1\textwidth}{!}{
\begin{tabular}{|c|c|c|c|}
	\hline Sampling rules & Convergence Rate Bound & Rate Bound Shown In & Parameters\\\hline Fixed, $p_{i}^{k} \equiv p_{i}$ & $1-\frac{\mu\gamma\sigma_p^2(S)}{2\|A\|^2}$ & Theorem \ref{main1} &$\theta=0$, $\beta_k=1.$\\
	\hline Max-distance & $(1-\frac{\mu\gamma\sigma_{\infty}^2(S)}{2\|A\|^2}  )$ & Theorem \ref{max distance}& $\theta=1$, $\beta_k=q.$\\
	\hline Proportional adaptive rule & $(1-\frac{\mu\gamma\sigma_{u}^2(S)}{\|A\|^2}  )$ & Theorem \ref{propsa}& $p_1 \propto g_{i}$,$\beta_k=1$ \\
	\hline Capped & $\left(1-\frac{\mu\gamma(\theta \sigma_{\infty}^{2}(S)+(1-\theta) \sigma_{p}^{2}( S))}{2\|A\|^2}\right)$ & Theorem \ref{capt}& $\beta_k=q$ \\
	\hline
	Sketch Motzkin & $\left(1-\frac{\mu \gamma\cdot\sigma^2_{p_1^k,\infty}(\beta_k,S)}{2\|A\|^2}  \right)$ & Theorem \ref{KMot}& $\beta_k=q$\\
	\hline
\end{tabular}}
\end{table}
\subsection{Convergence of adaptive SBP}
We now consider the adaptive sampling strategy. The main result is presented below.
\begin{theorem}\label{main2}
	Consider  Algorithm 2 with probabilities $p_1^k\in\Delta_q^{\dagger}$, and $p^{k}_2\in\Delta_q^{\dagger}$, $\beta_k\in[q]$. If $f$ is a $\mu$-strongly convex function such that the conditions in Lemma \ref{2016linear} holds, and the initialization $x^{0} \in \mathbb{R}^{n}$ and $x^{0}_{*} \in \partial f\left(x^{0}\right) \cap$
	$\mathcal{R}\left(A^{\top}\right)$, then the iterates in Algorithm 2 converge linearly in the sense that
	$$\mathbb{E}\left[D_{f}^{x^{k+1}_* }\left(x^{k+1}, \bar{x}\right)\right] \leqslant\left(1-\frac{\mu \gamma\left(\theta\sigma^2_{p_1^k,\infty}(\beta_k,S)+(1-\theta)\sigma^2_{p_1^k,p^{k}_2}(\beta_k,S)\right)}{2\|A\|^2}  \right)^{k+1}D_f^{x^0_*}(x^0,\overline{x}).$$
\end{theorem}
\noindent The main idea of the proof for the Theorem \ref{main2} is similar to that of Theorem \ref{main1} except \eqref{Le5} is changed when applying the adaptive sampling rule.
\begin{lemma}\label{Len} Assume the conditions in Theorem \ref{main2} are held, let the sequence $\{x^k\}$ be generated by  Algorithm 2. Then, we have
	\begin{eqnarray*}
	\mathbb{E}\left[g_{i_k}(x^k)\right]\geq\frac{\left(\theta\sigma^2_{p_1^k,\infty}(\beta_k,S)+(1-\theta)\sigma^2_{p_1^k,p^{k}_2}(\beta_k,S)\right)}{\|A\|^2}  \cdot\left\|A x^{k}-b\right\|^{2}.
	\end{eqnarray*}
\end{lemma}
\begin{proof}
	When $i_k$ is sampled with the adaptive rule in Algorithm 2.
$$	\begin{aligned}
		\mathbb{E}\left[g_{i_k}(x^k)\right]&=\sum_{\tau_k\sim p_1^k} \sum_{i_k\in\mathcal{W}_k}g_{i_k}(x^k)p_{3_{i_k}}^k\\
		&\geq  \sum_{\tau_k\sim p_1^k} \sum_{i_k \in \mathcal{W}_{k}}\left(\theta \max _{j\in\tau_k} g_{j}\left(x^{k}\right)+(1-\theta) \mathbb{E}_{j \sim p_2^k}\left[g_{j}\left(x^{k}\right)\right]\right) p_{3_{i_k}}^{k} \\
		&=\sum_{\tau_k\sim p_1^k}\left(  \theta \max _{j\in\tau_k} g_{j}\left(x^{k}\right)+(1-\theta) \mathbb{E}_{j \sim p_2^k}\left[g_{j}\left(x^{k}\right)\right]\right) \\
		& \geq\left(\theta \sigma_{p_1^k,\infty}^{2}( \beta_k,S)+(1-\theta) \sigma_{p_1^k,p_2^k}^{2}(\beta_k,S)\right)\left\|x^{k}-x^{*}\right\|^2 .
	\end{aligned}$$
	As a result, 	$$\mathbb{E}\left[g_{i_k}(x^k)\right]\geq\frac{\theta \sigma_{p_1^k,\infty}^{2}(\beta_k, S)+(1-\theta) \sigma_{p_1^k,p_2^k}^{2}(\beta_k,S)}{\|A\|^2}\left\|Ax^{k}-b\right\|^2.$$
\end{proof}
\noindent Then, by using the same idea of proving Theorem \ref{main1}, we can finish the proof of Theorem \ref{main2}.
\subsection{Special cases for adaptive sampling rule}
In this subsection, we will discuss four special cases of the adaptive sampling rule: max-distance, proportional to the sketched loss rule, capped adaptive rule, and sketch Motzkin rule. They can be generalized from Algorithm 2. For simplicity, we assume the conditions below are satisfied:$f$ is a $\mu$-strongly convex function such that the conditions in Lemma \ref{2016linear} holds, and the initialization $x^{0} \in \mathbb{R}^{n}$ and $x^{0}_{*} \in \partial f\left(x^{0}\right) \cap$
	$\mathcal{R}\left(A^{\top}\right)$. Results are summarised in Table \ref{special cases}.
\subsubsection{Max-distance}
\begin{table}
	\label{a3}
	\begin{tabular}{l}
		\hline
		Algorithm 3 Max-distance SBP method \\
		\hline 1: \textbf{Input}: $x^{0}=x^{0}_*=0 \in \mathbb{R}^{n}$, strongly convex function $f$, $A \in \mathbb{R}^{m \times n}, b \in \mathbb{R}^{m},$ and $S=\left[S_{1}, \ldots, S_{q}\right]$ \\
		$\quad\quad\quad\quad S_i\in\mathbb{R}^{m\times\tau},i=1,\cdots,q$\\
		2: for $k=0,1,2, \ldots$ do \\
		3: $\quad g_{i}\left(x^{k}\right)=\left\|A x^{k}-b\right\|_{H_{i}}$ for $i=1, \ldots, q,$ where $H_{i}:=S_i(S_{i}^{\top}AA^{\top}S_{i})^{\dagger}S_{i}^{\top}$ \\
        4: $\quad i_{k}=\arg \max _{i=1, \ldots, q} g_{i}\left(x^{k}\right)$ \\
        5: $\quad y^k\in\arg\min_{y\in\mathbb{R}^{\tau}} f^*(x_*^k-A^{\top}S_{i_k}y)+\langle S_{i_k}^{\top}b, y\rangle$\\
		6: $\quad x^{k+1}_*=x_*^k-A^{\top}S_{i_k}y^k$\\
		7: $\quad  x^{k+1}=\nabla f^*(x_*^{k+1})$\\
		8: output: last iterate $x^{k+1}$\\
		\hline	
	\end{tabular}
\end{table}
When $\theta=1$, $\beta_k=q$, the adaptive sampling rule becomes the max-distance\cite{Gower2019}. Theorem below provides a convergence guarantee for it.
\begin{theorem}\label{max distance}
Let $\{x^k\}$ be generated by Algorithm 3. Then, we have
	$$\mathbb{E}\left[D_{f}^{x^{k+1}_*}\left(x^{k+1}, \bar{x}\right)\right] \leqslant(1-\frac{\mu\gamma\sigma_{\infty}^2(S)}{2\|A\|^2}  )^{k+1}D_f^{x^0_*}(x^0,\overline{x}),$$
where $\sigma^2_{\infty}(S)$ is defined as in \eqref{spec2}.
\end{theorem}
\noindent Furthermore, we show that the convergence of the max-distance method is strictly faster than the method using fixed probability sampling rule.

\begin{theorem}Let $p \in \Delta_{q}^{\dagger}$ where $p_{i}>0$ for all $i=1, \ldots, q .$ Let $\sigma_{p}^{2}(S)$ be defined as in \eqref{spec1}, and define
\begin{eqnarray}\label{eta}
\eta := \frac{1}{\max _{i=1, \ldots, q} \sum_{j=1, j \neq i}^{q} p_{j}}>1.
\end{eqnarray}
Let $\{x^k\}$ be generated by Algorithm 3. Then, we have
$$
\mathbb{E}\left[D_{f}^{x^{k+1}_*}\left(x^{k+1}, \bar{x}\right) \mid x^k\right] \leq\left(1-\frac{\eta\mu\gamma \sigma_{p}^{2}( S)}{2\|A\|^2}\right)D_{f}^{x^{k}_*}\left(x^{k}, \bar{x}\right).
$$
\end{theorem}
\begin{proof}
 In Lemma \ref{S0}, recall that $g_{i_{k}}\left(x^{k+1}\right)=0$. Then, we  have
\begin{eqnarray}\label{ps1}
\mathbb{E}_{j \sim p}\left[g_{j}\left(x^{k+1}\right)\right] &=&\sum_{j=1, j \neq i_{k}}^{q} p_{j} g_{j}\left(x^{k+1}\right)\nonumber \\
& \leq&\left(\max _{j=1, \ldots, q} g_{j}\left(x^{k+1}\right)\right)\left(\sum_{j=1, j \neq i_{k}}^{q} p_{j}\right)\nonumber \\
& \leq&\left(\max _{j=1, \ldots, q} g_{j}\left(x^{k+1}\right)\right)\left(\max _{i=1, \ldots, q} \sum_{j=1, j \neq i}^{q} p_{j}\right)\nonumber \\
& =& \frac{\max _{j=1, \ldots, q} g_{j}\left(x^{k+1}\right)}{\eta}.
\end{eqnarray}

\noindent Then, by \eqref{EB1}, \eqref{Le5} and \eqref{ps1} , we deduce
\begin{eqnarray}
\frac{\max _{j=1, \ldots, q} g_{j}\left(x^{k+1}\right)}{\eta}\geq\mathbb{E}_{j\sim p}[g_{j}(x^k)]\geq\frac{\sigma^2_p(S)}{\|A\|^2}  \cdot\left\|Ax^{k}-b\right\|^{2}\geq \frac{\gamma\sigma^2_p(S)}{\|A\|^2}\cdot D_{f}^{x^k_{*}}(x^k, \overline{x}).\label{ps}
\end{eqnarray}
Last, by using \eqref{inequall},  we get
	$$\mathbb{E}\left[D_{f}^{x^{k+1}_*}\left(x^{k+1}, \bar{x}\right) \mid x^k\right] \leqslant(1-\frac{\eta\mu\gamma\sigma_p^2(S)}{2\|A\|^2}  )D_f^{x^k_*}(x^k,\overline{x}),$$
which is strictly less than the convergence rate of the method with fixed probability sampling rule in Theorem \ref{main1}.
\end{proof}
\subsubsection{Proportional to the sketched loss}
When $\beta_k=1$, $p_1=\frac{g(x^k)}{\|g(x^k)\|_1} $, the adaptive sampling rule \cite{Gower2019} can be achieved which is a kind of the proportional to the sketched loss adaptive rule where indices are sampled with the probabilities proportional to the sketched loss values. For this sampling rule, we can also derive a convergence rate, which is as least twice faster than the method with the uniform sampling rule.
\begin{table}
	\label{a4}
	\begin{tabular}{l}
		\hline
		Algorithm 4 Proportional to the sketched loss SBP method\\
		\hline 1: \textbf{Input}: $x^{0}=x^{0}_*=0 \in \mathbb{R}^{n}$, srongly convex function $f$, $A \in \mathbb{R}^{m \times n}, b \in \mathbb{R}^{m},$ and $S=\left[S_{1}, \ldots, S_{q}\right]$ \\
		$\quad\quad\quad\quad S_i\in\mathbb{R}^{m\times\tau},i=1,\cdots,q$\\
		2: for $k=0,1,2, \ldots$ do \\
		3: $\quad g_{i}\left(x^{k}\right)=\left\|A x^{k}-b\right\|_{H_{i}}$ for $i=1, \ldots, q,$ where $H_{i}:=S_i(S_{i}^{\top}AA^{\top}S_{i})^{\dagger}S_{i}^{\top}$ \\
		4: $\quad i_{k}\sim p^k$, $p^k=\frac{g(x^k)}{\|g(x^k)\|_1}$ \\
	    5: $\quad y^k\in\arg\min_{y\in\mathbb{R}^{\tau}} f^*(x_*^k-A^{\top}S_{i_k}y)+\langle S_{i_k}^{\top}b, y\rangle$\\
	    6: $\quad x^{k+1}_*=x_*^k-A^{\top}S_{i_k}y^k$\\
		7: $\quad  x^{k+1}=\nabla f^*(x_*^{k+1})$\\
		8: output: last iterate $x^{k+1}$\\
		\hline	
	\end{tabular}
\end{table}
\begin{theorem}\label{propsa}
	Consider Algorithm 4 with $p^{k}=\frac{g\left(x^{k}\right)}{\left\|g\left(x^{k}\right)\right\|_{1}} .$ Define $u=\left(\frac{1}{q}, \ldots, \frac{1}{q}\right) \in \Delta_{q}$. Then, we have
$$
\mathbb{E}\left[D_{f}^{x^{k+1}_*}\left(x^{k+1}, \bar{x}\right) \right] \leq\left(1-\frac{\mu\gamma \left(1+q^{2} \mathbb{V} \mathbb{A} \mathbb{R}_{i \sim u}\left[p_{i}^{k}\right]\right) \sigma_{u}^{2}( S)}{2\|A\|^2}\right)^{k}D_{f}^{x^{0}_*}\left(x^{1}, \bar{x}\right), k\geq0,
$$
where $\mathbb{V A} \mathbb{R}_{i \sim u}[\cdot]$ denotes the variance taken with respect to the uniform distribution
$$
\mathbb{V A R}_{i \sim u}\left[v_{i}\right] \stackrel{\triangle}{=} \frac{1}{q} \sum_{i=1}^{q}\left(v_{i}-\frac{1}{q} \sum_{s=1}^{q} v_{s}\right)^2, \quad \forall v \in \mathbb{R}^{q}.
$$
Moreover, we have
$$
\mathbb{E}\left[D_{f}^{x^{k+1}_*}\left(x^{k+1}, \bar{x}\right)\right] \leq\left(1-\frac{\mu\gamma\sigma_{u}^{2}( S)}{\|A\|^2}\right)^kD_{f}^{x^{0}_*}\left(x^{1}, \bar{x}\right).
$$
\end{theorem}
\begin{proof}
Note that for $i \sim u$.  We derive that
$$
\begin{array}{l}
\qquad \mathbb{V A R}_{i\sim u}\left[g_{i}\left(x^{k}\right)\right]=\mathbb{E}_{i\sim u}\left[\left(g_{i}\left(x^{k}\right)\right)^{2}\right]-\left(\mathbb{E}_{i\sim u}\left[g_{i}\left(x^{k}\right)\right]\right)^{2}=\frac{1}{q} \sum_{i=1}^q\left(g_{i}\left(x^{k}\right)\right)^{2}-\frac{1}{q^{2}}\left(\sum_{i=1}^{q} g_{i}\left(x^{k}\right)\right)^{2} . \\
\text{Thus, we  have}, \\
\qquad \begin{aligned}
\mathbb{E}_{i \sim p^{k}}\left[g_{i}\left(x^{k}\right)\right] &=\sum_{i=1}^{q} p_{i}^{k} g_{i}\left(x^{k}\right) \\
&=\sum_{i=1}^{q} \frac{\left(g_{i}\left(x^{k}\right)\right)^{2}}{\sum_{i=1}^{q} g_{i}\left(x^{k}\right)} \\
&=\frac{q \mathbb{V} \mathbb{A} \mathbb{R}_{i\sim u}\left[g_{i}\left(x^{k}\right)\right]+\frac{1}{q}\left(\sum g_{i}\left(x^{k}\right)\right)^{2}}{\sum_{i=1}^{q} g_{i}\left(x^{k}\right)} \\
&=\left(q^{2} \mathbb{V} \mathbb{A} \mathbb{R}_{i\sim u}\left[\frac{g_{i}\left(x^{k}\right)}{\sum_{i=1}^{q} g_{i}\left(x^{k}\right)}\right]+1\right) \frac{1}{q} \sum_{i=1}^{q} g_{i}\left(x^{k}\right)\\
&\geq\frac{\left(q^{2} \mathbb{V} \mathbb{A} \mathbb{R}_{i\sim u}\left[p_i^k\right]+1\right)\sigma_{u}^{2}(S)}{\left\|A\right\|^2}\left\|Ax^{k}-b\right\|^{2},\end{aligned}
\end{array}
$$
where the last inequality is deduced by $p^{k}_i=\frac{g_i\left(x^{k}\right)}{\sum_{i=1}^{q} g_{i}\left(x^{k}\right)}$, the inequality \eqref{spectrum2}, and $\left\|A x^{k}-b\right\|\leqslant\|A\| \cdot\left\|x^{k}-\bar{x}\right\|.$ Then, by using \eqref{EB1} and \eqref{inequall}, we have
\begin{eqnarray}\label{proportional1}
\mathbb{E}\left[D_{f}^{x^{k+1}_*}\left(x^{k+1}, \bar{x}\right)  \mid x^k\right]\leq\left(1-\frac{\mu\gamma\left(1+q^{2} \mathbb{V} \mathbb{A} \mathbb{R}_{i\sim u}\left[p^{k}_i\right]\right) \sigma_{u}^{2}( S)}{2\|A\|^2}\right)D_{f}^{x^{k}_*}\left(x^{k}, \bar{x}\right).
\end{eqnarray}

It remains to show the second part. By Lemma \ref{S0}, we know $g_{i_k}(x^{k+1})=0$. Then, $p_{i_{k}}^{k+1}=0$ can be derived. Thus,
$$
\begin{aligned}
\mathbb{V A R}_{i\sim u}\left[p_{i}^{k+1}\right] & = \frac{1}{q} \sum_{i=1}^{q}\left(p_{i}^{k+1}-\frac{1}{q} \sum_{j=1}^{q} p_{j}^{k+1}\right)^{2} \\
&=\frac{1}{q} \sum_{i=1}^{q}\left(p_{i}^{k+1}-\frac{1}{q}\right)^{2} \geq \frac{1}{q}\left(p_{i_{ k}}^{k+1}-\frac{1}{q}\right)^{2}=\frac{1}{q^{2}}.
\end{aligned}
$$
Combining with \eqref{proportional1}, we get
$$
\mathbb{E}\left[D_{f}^{x^{k+1}_*}\left(x^{k+1}, \bar{x}\right) \mid x^{k}\right] \leq\left(1-\frac{\mu\gamma\sigma_{u}^{2}( S)}{\|A\|^2}\right)D_{f}^{x^{k}_*}\left(x^{k}, \bar{x}\right), k\geq1.
$$
This completes the proof.
\end{proof}
\subsubsection{Capped sampling rule}
When $\beta_k=q$, we obtain the capped SBP method, introduced in Algorithm 5.  This sampling rule is originally suggested in [3] for the randomized Kaczmarz method. Below, we provide the convergence guarantees for Algorithm 5.
\begin{table}
	\label{a5}
	\begin{tabular}{l}
		\hline
		Algorithm 5 Capped SBP method \\
		\hline 1: \textbf{Input}: $x^{0}=x^{0}_*=0 \in \mathbb{R}^{n}$, strongly convex function $f$, $A \in \mathbb{R}^{m \times n}, b \in \mathbb{R}^{m},$ $p \in \Delta_{q}^{\dagger},$\\ $\quad\quad\quad\quad$ and $S=\left[S_{1}, \ldots, S_{q}\right],S_i\in\mathbb{R}^{m\times\tau},i=1,\cdots,q$ \\
		2: for $k=0,1,2, \ldots$ do \\
		3: $\quad g_{i}\left(x^{k}\right)=\left\|A x^{k}-b\right\|_{H_{i}}$ for $i=1, \ldots, q,$ where $H_{i}:=S_i(S_{i}^{\top}AA^{\top}S_{i})^{\dagger}S_{i}^{\top}$ \\
		4: $\quad\mathcal{W}_{k}=\left\{i \mid g_{i}\left(x^{k}\right) \geq \theta \max _{j=1, \ldots, q} g_{j}\left(x^{k}\right)+(1-\theta) \mathbb{E}_{j \sim p}\left[g_{j}\left(x^{k}\right)\right]\right\}$ \\
		5: $\quad i_k\sim p^k$, where $p^{k} \in \Delta_{q}^{\dagger}$ such that $\textbf{supp}\left(p^{k}\right) \subset \mathcal{W}_{k}$\\
		6: $\quad y^k\in\arg\min_{y\in\mathbb{R}^{\tau}} f^*(x_*^k-A^{\top}S_{i_k}y)+\langle S_{i_k}^{\top}b, y\rangle$\\
		7: $\quad x^{k+1}_*=x_*^k-A^{\top}S_{i_k}y^k$\\
		8: $\quad  x^{k+1}=\nabla f^*(x_*^{k+1})$\\
		9: output: last iterate $x^{k+1}$\\
		\hline	
	\end{tabular}
\end{table}
\begin{theorem}\label{capt}
Consider Algorithm 5 and let $p \in \Delta_{q}^{\dagger}$ be a fixed reference probability and $\theta \in[0,1] .$ Define
$$
\mathcal{W}_{k}=\left\{i \mid g_{i}\left(x^{k}\right) \geq \theta \max _{j=1, \ldots, q} g_{j}\left(x^{k}\right)+(1-\theta) \mathbb{E}_{j \sim p}\left[g_{j}\left(x^{k}\right)\right]\right\}.
$$
Then, we have
$$
\mathbb{E}\left[D_{f}^{x^{k+1}_*}\left(x^{k+1}, \bar{x}\right)\right] \leq\left(1-\frac{\mu\gamma(\theta \sigma_{\infty}^{2}(S)+(1-\theta) \sigma_{p}^{2}( S))}{2\|A\|^2}\right)^{k+1}D_{f}^{x^{0}_*}\left(x^{0}, \bar{x}\right).
$$
\end{theorem}
The proof of Theorem \ref{capt} is similar to that of Theorem \ref{main2}; we omit it here. Note that, for the capped SBP method, in each iteration $p_{i}^{k}$ are zero for all indices that are not in the set $\mathcal{W}_{k}$, which contains indices whose sketched losses are lager than the weighted sum of the maximal sketched loss and $\mathbb{E}_{i \sim p}\left[g_{i}\left(x^{k}\right)\right].$  The parameter $\theta$ can control the aggressive of the method. When $\theta=1$, capped SBP method is equivalent to the max-distance SBP method. But when $\theta=0$, the convergence rate of the sampling rule is equivalent to that of the non-adaptive SBP.

By using the relationships between different spectral constants in Lemma \ref{relationshipspectral}, we conclude that the convergence rate of the SBP method with the capped sampling rule is not slower than that of the non-adaptive SBP method.

\subsubsection{Sketch Motzkin SBP method}
When $\theta=1$, we can derive a sampling rule called the Sketched Motzkin method which can be viewed as  a generalization of the sampling Kaczmarz-Motzkin method in \cite{haddock2019greed}. The convergence rate of this method is shown in Theorem \ref{KMot}, which can be directly derived from Theorem \ref{main2}. At the same time, we can also derive a convergence rate below and show that it can be as least faster than the SBP method with the uniform sampling rule.
\begin{theorem}\label{KMot}
	 Let$\quad\tau_k\in\binom{[q]}{\beta_k}\sim p_1^k$ $,p_1^k \in \Delta_{|\binom{[q]}{\beta_k}|}$, and $\beta_k\leq q$. The iterates in Algorithm 6 converge linearly in the sense that
	$$\mathbb{E}\left[D_{f}^{x^{k+1}_* }\left(x^{k+1}, \bar{x}\right)\right] \leqslant\left(1-\frac{\mu \gamma\cdot\sigma^2_{p_1^k,\infty}(\beta_k,S)}{2\|A\|^2}  \right)^{k+1}D_f^{x^0_*}(x^0,\overline{x}).$$
	Moreover, define $u_q=\left(\frac{1}{q}, \ldots, \frac{1}{q}\right) \in \Delta_{q}^{\dagger}$,
		$$\mathbb{E}\left[D_{f}^{x^{k+1}_* }\left(x^{k+1}, \bar{x}\right)\right] \leqslant\left(1-\frac{\mu \gamma\cdot\sigma^2_{u_{|\binom{[q]}{\beta_k}|},\infty}(\beta_k,S)}{2\|A\|^2}  \right)^{k+1}D_f^{x^0_*}(x^0,\overline{x})\leqslant\left(1-\frac{\mu \gamma\cdot\sigma^2_{u_q}(S)}{2\|A\|^2}  \right)^{k+1}D_f^{x^0_*}(x^0,\overline{x}).$$
\end{theorem}
\begin{proof}
We only prove the second part. Recall \eqref{spectrum4}, the relationship below can be held.
	\begin{eqnarray*}\label{spectrum5}
			\sigma^2_{u_{|\binom{[q]}{\beta_k}|},\infty}(\beta_k,S):= \min _{v \notin \mathrm{Null(A)}}\mathbb{E}_{\tau_k\sim u_{|\binom{[q]}{\beta_k}|}}\left[ \max _{i\in\tau_k} \frac{\left\| v\right\|_{Z_i}^{2}}{\|v\|^{2}}\right]&\geq&\min _{v \notin \mathrm{Null(A)}}\mathbb{E}_{\tau_k\sim u_{|\binom{[q]}{\beta_k}|}}\left[ \mathbb{E} _{i\sim u_{\beta_k},i\in \tau_k} \frac{\left\| v\right\|_{Z_i}^{2}}{\|v\|^{2}}\right]\\
			&=&\min _{v \notin \mathrm{Null(A)}}\mathbb{E} _{i\sim u_{q}} \frac{\left\| v\right\|_{Z_i}^{2}}{\|v\|^{2}}=	\sigma^2_{u_{q}}(S).
			\end{eqnarray*}
Then, combining with the results in the first part of Theorem \ref{KMot}, we finish the proof.
\end{proof}

Overall, by Lemma \ref{relationshipspectral} and the rates of the convergence in Table \ref{special cases}, we can find that the max-distance SBP method has the fastest convergence rate. Next is the sketched Motzkin SBP method , capped SBP method and proportional to the sketched loss SBP method. Last is the non-adaptive SBP method utilizing the uniform sampling rule. In the next section, we will apply numerical tests to verify the theoretical results above. Furthermore, it should be noted that the SBP method with adaptive sampling rule needs more computational cost in each iteration although the method utilizing them usually achieve a faster converge rate.
\begin{table}
	\label{a2}
	\begin{tabular}{l}
		\hline
		Algorithm 6 Sketch Motzkin SBP method\\
		\hline ~1: \textbf{Input}: $x^{0}=x^{0}_*=0 \in \mathbb{R}^{n}$, strongly convex function $f$, $A \in \mathbb{R}^{m \times n}, b \in \mathbb{R}^{m},$ $S=\left[S_{1}, \ldots, S_{q}\right],$\\
		$\quad \quad \quad\quad S_i\in\mathbb{R}^{m\times\tau},i=1,\cdots,q$ \\
		~2: for $k=0,1,2, \ldots$ do \\
		~3: $\quad\tau_k\in\binom{[q]}{\beta_k}\sim p_1^k$ $,p_1^k \in \Delta_{|\binom{[q]}{\beta_k}|}$, and $\beta_k\leq q$ \\
		~4: $\quad i_{k}=\arg \max _{i\in\tau_k} g_{i}\left(x^{k}\right)$ \\
		~5: $\quad y^k\in\arg\min_{y\in\mathbb{R}^{\tau}} f^*(x_*^k-A^{\top}S_{i_k}y)+\langle S_{i_k}^{\top}b, y\rangle$\\
		~6: $\quad x^{k+1}_*=x_*^k-A^{\top}S_{i_k}y^k$\\
		~7: $\quad  x^{k+1}=\nabla f^*(x_*^{k+1})$\\
		~8: \textbf{Output}: last iterate $x^{k+1}$\\
		\hline	
	\end{tabular}
\end{table}
%
%
%
%
%
\section{Applications}
In this section, we introduce some applications of the SBP method to the Kaczmarz method and the sparse Kaczmarz method.
\subsection{Randomized Kaczmarz method}
Set $f(x)=\frac{1}{2}\|x\|_2^2$ and take the sketching matrices $S_{i}=e_{i}$ for $i=1, \ldots, m$, where $e_{i} \in \mathbb{R}^{n}$ is the $i{\text {th }}$ coordinate vector. In this setting, $f^*(x)=\frac{1}{2}\|x\|^2$ and hence the resulting method corresponds to the randomized Kaczmarz method, read as
\begin{eqnarray*}
	x^{k+1}=x^{k}-\frac{\left\langle A_{i_k:}^{\top}, x^{k}\right\rangle-b_{i_k}}{\left\|A_{i_k:}\right\|_{2}^{2}} \cdot A_{i_k:}^{\top}, i_k\sim p^k.
\end{eqnarray*}
 Note that
 \begin{eqnarray}
&H_i:=e_i(e_i^{\top}AA^{\top}e_i)^{\dag}e_i^{\top}=\frac{e_ie_i^{\top}}{\|A_{i:}\|^2},&\\
&Z_i:=A^{\top}H_iA=\frac{A_{i:}^{\top}A_{i:}}{\|A_{i:}\|^2}.&
\end{eqnarray}
Hence, the sketched loss becomes
\begin{eqnarray*}
	g_{i_k}(x^k):=\|x^k-\overline{x}\|^2_{Z_{i_k}}=\frac{\|\langle A_{i_k:}^{\top},x^k\rangle-b_{i_k}\|^2}{\|A_{i_k:}\|^2}.
\end{eqnarray*}

Using the results in Section 4, we obtain a group of convergence rate bounds for Kaczmarz methods, equipped with different sampling strategies; the details are summarized in Table \ref{ttt1}.
\begin{table}
\caption{Summary of convergence rate bounds and parameters choosing for different sampling strategies for the randomized Kaczmarz algorithm. Here, $\eta=1 / \max _{i=1, \ldots, m} \sum_{j=1, j \neq i}^{m} p_{i}$ as defined in $\eqref{eta}$, $P=\operatorname{diag}\left(p_{1}, \ldots, p_{m}\right)$ is
	a matrix of arbitrary fixed probabilities, and $\overline{A}$ is the normalization of $A$. The sampling size is $\beta_k$, and the number of rows and columns in the matrix $A$ are $m$ and $n$ respectively, $\gamma_k:=\frac{\sum_{\tau_k\in\binom{[m]}{\beta_k}}\|A_{\tau_k}\mathbf{x}_{k}-\mathbf{b}_{\tau_k}\|^2_2}{\sum_{\tau_k\in\binom{[m]}{\beta_k}}\|A_{\tau_k}\mathbf{x}_{k}-\mathbf{b}_{\tau_k}\|^2_{\infty}}\leq\beta_k,$  $\theta\in[0,1]$, and $\sigma^{\dagger}(\cdot)$ is the smallest nonzero singular value of a matrix.}\label{ttt1}
\centering
\resizebox{0.9\textwidth}{!}{
\begin{tabular}{|c|c|c|c|}
	\hline Sampling Strategy & Convergence Rate Bound & Rate Bound In & Parameter choosing \\
	\hline \multirow{2}{*}{Uniform} & \multirow{2}{*}{$1-\frac{1}{m} \sigma_{\min }^{\dagger}\left(\overline{A}^{\top} \overline{A}\right)$} & \multirow{2}{*}{Theorem $1$\cite{Maxdis}}& $f(x)=\frac{1}{2}\|x\|_2^2, S_i=e_i$,\\
	&&&$p_{1_i}=\frac{1}{m},$ $ \theta=0$, $\beta_k=1.$ \\
	
	\hline\multirow{2}{*}{$p_{1_i} \propto\left\|A_{i:}\right\|_{2}^{2}$} &\multirow{2}{*}{$1-\frac{\sigma_{\min }^{\dagger}\left(A^{\top} A\right)}{\|A\|_{F}^{2}}$} & \multirow{2}{*}{ Theorem 7\cite{2009AVershyin}}& $f(x)=\frac{1}{2}\|x\|_2^2, S_i=e_i$\\
	&&& $p_{1_i}=\frac{\|A_{i:}\|_2^2}{\|A\|_{F}^2},\theta=0$, $\beta_k=1.$ \\
	\hline \multirow{2}{*}{Max-distance} & \multirow{2}{*}{$1-\frac{\sigma_{\min }^{\dagger}\left(\overline{A}^{\top} \overline{A}\right)}{\gamma_{k}}$} & \multirow{2}{*}{Theorem 1\cite{haddock2019greed}} & $f(x)=\frac{1}{2}\|x\|_2^2, S_i=e_i$\\
	&&&$\theta=1$, $\beta_k=m.$ \\
	\hline\multirow{2}{*}{\text{Proportional}} & \multirow{2}{*}{$1-\frac{2}{m} \sigma_{\min }^{\dagger}\left(\overline{A}^{\top} \overline{A}\right)$} &\multirow{2}{*}{Theorem 11\cite{Gower2019}} & $f(x)=\frac{1}{2}\|x\|_2^2, S_i=e_i$\\
	&&&$p_1 \propto g_{i}$,$\beta_k=1.$ \\
	\hline \multirow{2}{*}{Capped} & \multirow{2}{*}{$1-(\theta \eta+1) \sigma_{\min }^{\dagger}\left(\overline{A}^{\top} P \overline{A}\right)$} & \multirow{2}{*}{Theorem 13\cite{bai2018on}} & $f(x)=\frac{1}{2}\|x\|_2^2, S_i=e_i$\\
	&&&$\beta_k=m$\\
	\hline Sampling & \multirow{2}{*}{$1-\frac{\beta_{k}\sigma_{\min }^{\dagger}\left(\overline{A}^{\top} \overline{A}\right)}{\gamma_{k}m}$} & \multirow{2}{*}{Theorem 1\cite{haddock2019greed}} & $f(x)=\frac{1}{2}\|x\|_2^2,S_i=e_i$\\
	Kaczmarz-Motzkin&&&$p_1\sim\frac{\left\|A_{\tau_k:}\right\|^{2}}{\sum_{\tau \in\binom{[m]}{\beta_k}}\|A_{\tau,:}\|^2},\theta=1$\\
	\hline
\end{tabular}
}
\end{table}
\subsection{Randomized sparse Kaczmarz method}
Set $f(x)=\frac{1}{2}\|x\|_2^2+\lambda\|x\|_1$ and take other parameters of the randomized Kaczmarz method. Then, we recover the randomized sparse Kaczmarz method. In this setting, $f^*(x)=\frac{1}{2}\|S_{\lambda}(x)\|^2,$ where $S_\lambda(x)$ is the soft-thresholding operator, defined by
$$S_\lambda(x)=\max(|x|-\lambda,0)\cdot\text{sign}(x).$$
Now, the randomized sparse Kaczmarz method reads as
\begin{eqnarray*}
x^{k+1}_{*}&=&x^{k}_{*}-t_k \cdot A_{i_k:},i_k\sim p^k \\
x^{k+1}&=&S_{\lambda}\left(x^{k+1}_{*}\right),
\end{eqnarray*}
where $t_k$ is the dual stepsize, called inexact step if $t_{k}=\left\langle A_{i_{k},:}^{\top}, x_{k}\right\rangle-b_{i_{k}},$ and if
\begin{eqnarray}\label{subo}
t_{k}\in
\operatorname{argmin}_{t \in \mathbb{R}} f^{*}\left(x_{k}^{*}-t \cdot A_{i_{k},:}\right)+t \cdot b_{i_{k}},
\end{eqnarray}
which is called the exact step. The method to solve this minimization subproblem \eqref{subo} above was analyzed in \cite{schopfer2016linear}.
Again, by using the results in Section 4, we obtain a group of  convergence rate bounds for the randomized sparse Kaczmarz methods, summarized in Table \ref{ttt2}.
\begin{table}
	\caption{Summary of convergence rate bounds and parameter choosing for different sampling strategies for the randomized Kaczmarz algorithm. Here, $\eta=1 / \max _{i=1, \ldots, m} \sum_{j=1, j \neq i}^{m} p_{i}$ as defined in Equation $(\ref{eta}), P=\operatorname{diag}\left(p_{1}, \ldots, p_{m}\right)$ is
		a matrix of arbitrary fixed probabilities, $\overline{A}$ is the normalization of $A$. $\lambda$ is the $l_1$ regularizer. The sampling size is $\beta_k$, and the number of rows and columns in the matrix $A$ are $m$ and $n$ respectively, $\gamma_k:=\frac{\sum_{\tau_k\in\binom{[m]}{\beta_k}}\|A_{\tau_k}\mathbf{x}_{k}-\mathbf{b}_{\tau_k}\|^2_2}{\sum_{\tau_k\in\binom{[m]}{\beta_k}}\|A_{\tau_k}\mathbf{x}_{k}-\mathbf{b}_{\tau_k}\|^2_{\infty}}\leq\beta_k,$ and $|\hat{x}|_{\min }$ is the smallest nonzero element of $|\hat{x}|$.}\label{ttt2}
		\resizebox{\textwidth}{!}{
	\begin{tabular}{|c|c|c|c|}
		\hline Sampling Strategy & Convergence Rate Bound & Rate Bound In & Parameter choosing \\
		\hline \multirow{2}{*}{Uniform} & \multirow{2}{*}{$1-\frac{1}{2m} \sigma_{\min }^{\dagger}\left(\overline{A}^{\top} \overline{A}\right)\frac{|\hat{x}|_{\min }}{|\hat{x}|_{\min }+2 \lambda}$} & \multirow{2}{*}{Theorem 1\cite{2016Linear}} & $f(x)=\frac{1}{2}\|x\|_2^2+\lambda\|x\|_1, S_i=e_i$,\\
		&&&$p_{1_i}=\frac{1}{m},$ $ \theta=0$, $\beta_k=1.$ \\
		
		\hline\multirow{2}{*}{$p_{1_i} \propto\left\|A_{i:}\right\|_{2}^{2}$} &\multirow{2}{*}{$1-\frac{\sigma_{\min }^{\dagger}\left(A^{\top} A\right)}{2\|A\|_{F}^{2}}\frac{|\hat{x}|_{\min }}{|\hat{x}|_{\min }+2 \lambda}$} & \multirow{2}{*}{Theorem 1\cite{2016Linear} }& $f(x)=\frac{1}{2}\|x\|_2^2+\lambda\|x\|_1, S_i=e_i$\\
		&&& $p_{1_i}=\frac{\|A_{i:}\|_2^2}{\|A\|_{F}^2},\theta=0$, $\beta_k=1.$ \\
		\hline \multirow{2}{*}{Max-distance} &\multirow{2}{*}{ $1-\frac{\sigma_{\min }^{\dagger}\left(\overline{A}^{\top} \overline{A}\right)}{2\gamma_{k}}\frac{|\hat{x}|_{\min }}{|\hat{x}|_{\min }+2 \lambda}$} & \multirow{2}{*}{Section 4.4.1} & $f(x)=\frac{1}{2}\|x\|_2^2+\lambda\|x\|_1, S_i=e_i$\\
		&&&$\theta=1$, $\beta_k=m.$ \\
			\hline\multirow{2}{*}{\text{Proportional}} & \multirow{2}{*}{$1-\frac{1}{m} \sigma_{\min }^{\dagger}\left(\overline{A}^{\top} \overline{A}\right)\frac{|\hat{x}|_{\min }}{|\hat{x}|_{\min }+2 \lambda}$} &\multirow{2}{*}{Section 4.4.2} & $f(x)=\frac{1}{2}\|x\|_2^2+\lambda\|x\|_1, S_i=e_i$\\
		&&&$p_1 \propto g_{i}$,$\beta_k=1.$ \\
			\hline \multirow{2}{*}{Capped} & \multirow{2}{*}{$1-\frac{(\theta \eta+1) \sigma_{\min }^{\dagger}\left(\overline{A}^{\top} P \overline{A}\right)}{2}\frac{|\hat{x}|_{\min }}{|\hat{x}|_{\min }+2 \lambda}$} & \multirow{2}{*}{Section 4.4.3} & $f(x)=\frac{1}{2}\|x\|_2^2, S_i=e_i$\\
	&&&$\beta_k=m$\\
		\hline Sampling & \multirow{2}{*}{$1-\frac{\beta_{k}\sigma_{\min }^{\dagger}\left(\overline{A}^{\top} \overline{A}\right)}{2\gamma_{k}m}\frac{|\hat{x}|_{\min }}{|\hat{x}|_{\min }+2 \lambda}$} & \multirow{2}{*}{Theorem 2\cite{Yuan2021}} & $f(x)=\frac{1}{2}\|x\|_2^2+\lambda\|x\|_1,S_i=e_i$\\
	Kaczmarz-Motzkin&&&$p_1\sim\frac{\left\|A_{\tau_k:}\right\|^{2}}{\sum_{\tau \in\binom{[m]}{\beta_k}}\|A_{\tau:}\|^2},\theta=1$\\
	\hline
	\end{tabular}}
\end{table}
\section{Numerical performance}
\subsection{Experimental setup}
In this section, we will testify the performance of the sparse Kaczmarz methods with different sampling strategies: uniform rule\cite{2009AVershyin}, max-distance\cite{griebel2012greedy}, proportional to the sketched loss adaptive rule\cite{Gower2015}, capped adaptive sampling rule with $\theta=0.5$\cite{bai2018on,bai20181on}, and the sampling Kaczmarz-Motzkin rule with $\beta=m/2$\cite{haddock2019greed,Yuan2021}. While for the Kaczmarz method, \cite{Gower2019} has made a large number of numerical tests to demonstrate that the max-distance sampling rule can have a faster convergence rate than other sampling rules. From Figure \ref{figure01}, we can find that the randomized Kaczmarz method equipped with sampling Kaczmarz-Motzkin rule has a little slower convergence speed than the max-distance rule, but it actually demands much less computational cost when $m$ is large. This observation is very important in practice when dealing with big data. In comparison, for the sparse Kaczmarz method, the results of the tests in this paper demonstrate that the method applying the sampling Kaczmarz-Motzkin rule demands the least computational cost to achieve a given error than other sampling rules.
\begin{figure}
	\begin{subfigure}{0.5\textwidth}
		\includegraphics[width=1\textwidth]{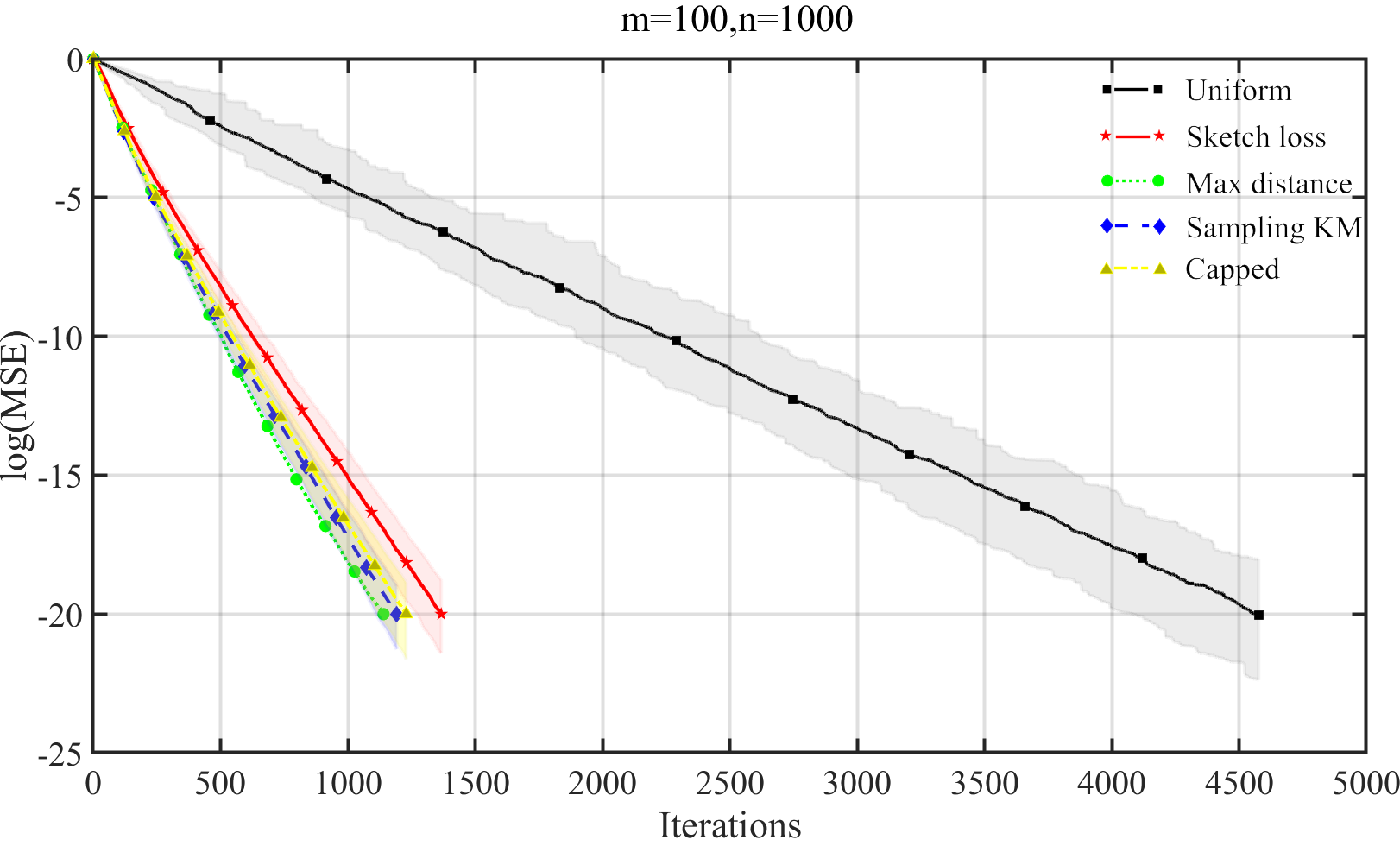}
		\caption{}
	\end{subfigure}
	\begin{subfigure}{0.5\textwidth}
		\includegraphics[width=1\textwidth]{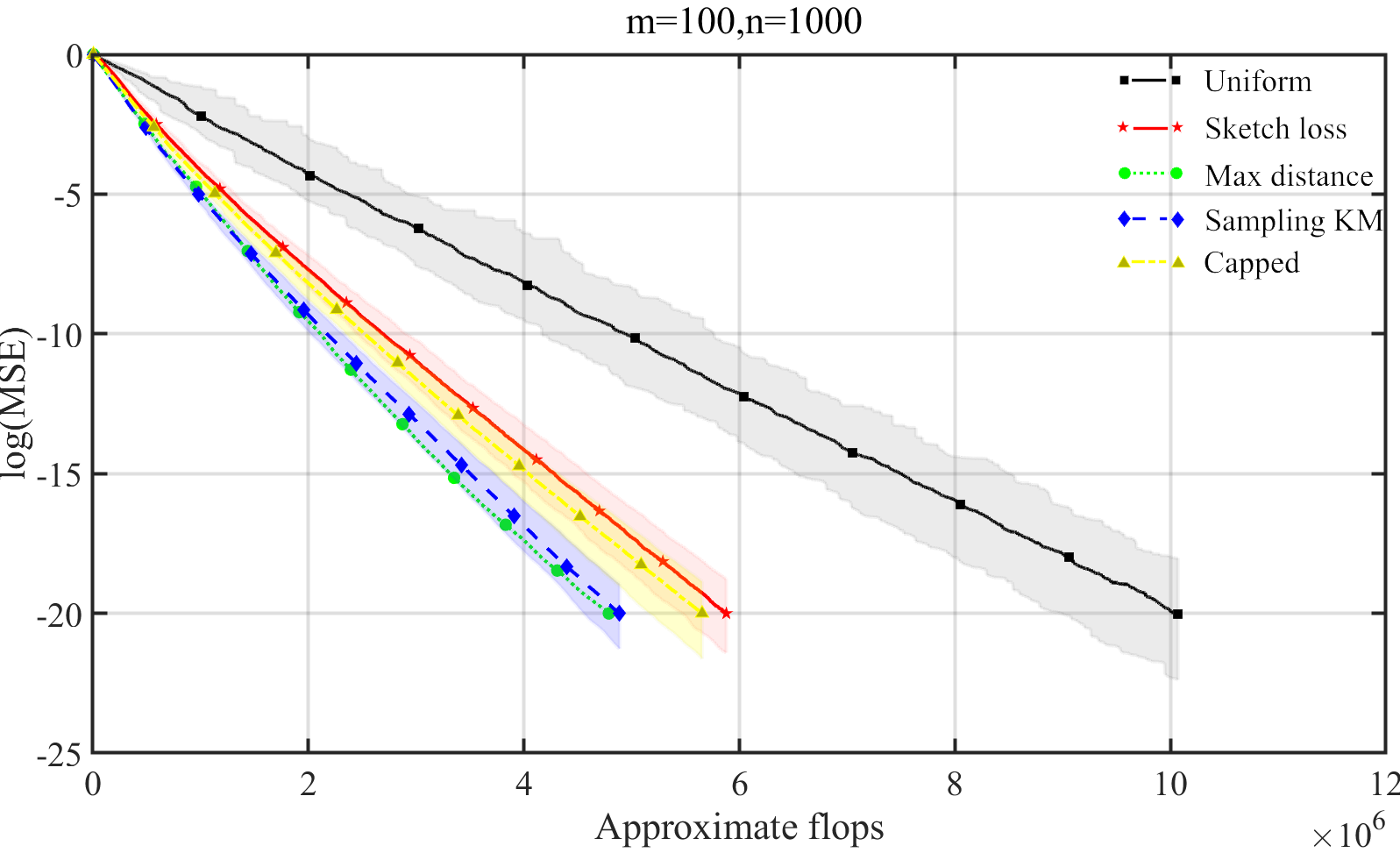}
		\caption{}
	\end{subfigure}
	\begin{subfigure}{0.5\textwidth}
		\includegraphics[width=1\textwidth]{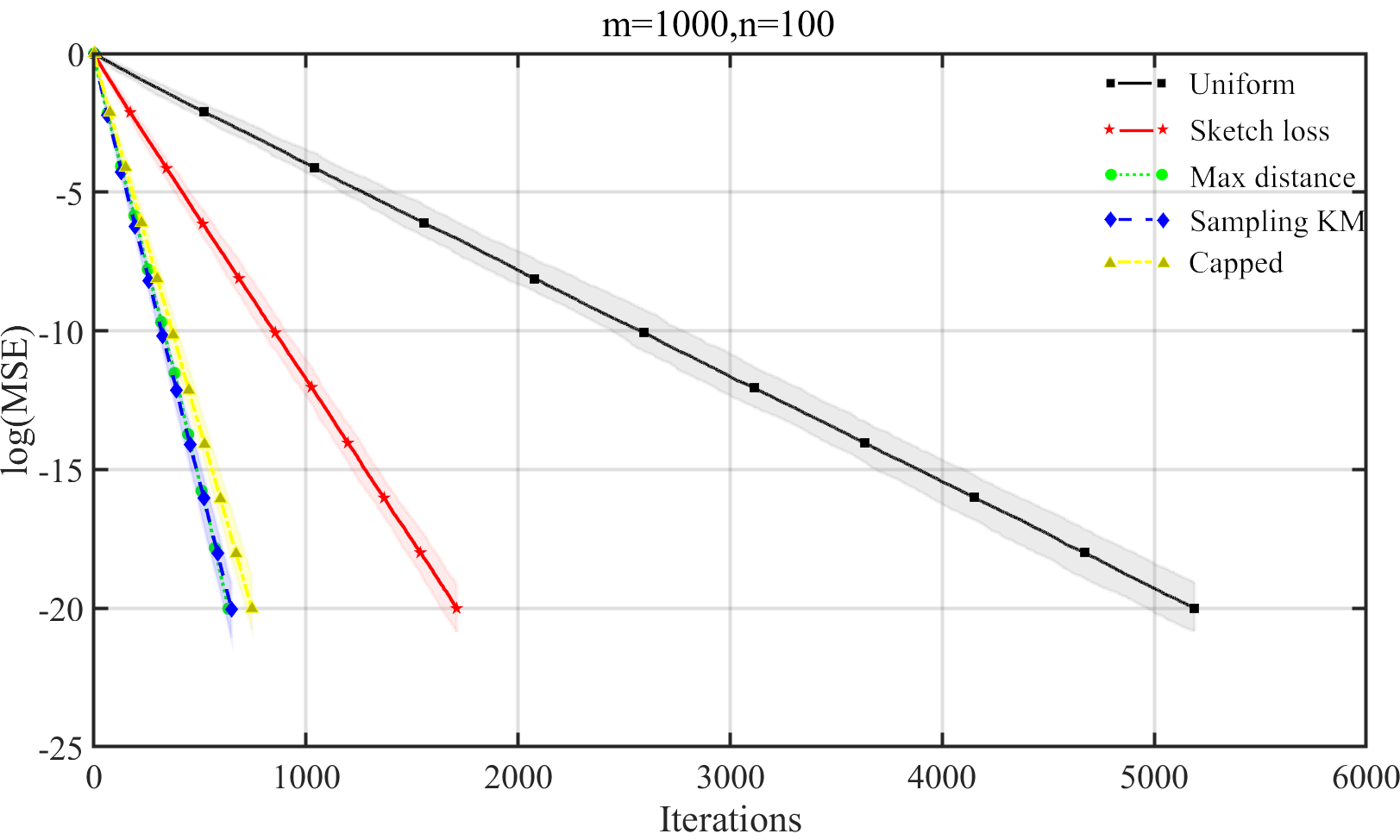}
		\caption{}
	\end{subfigure}
\begin{subfigure}{0.5\textwidth}
	\includegraphics[width=1\textwidth]{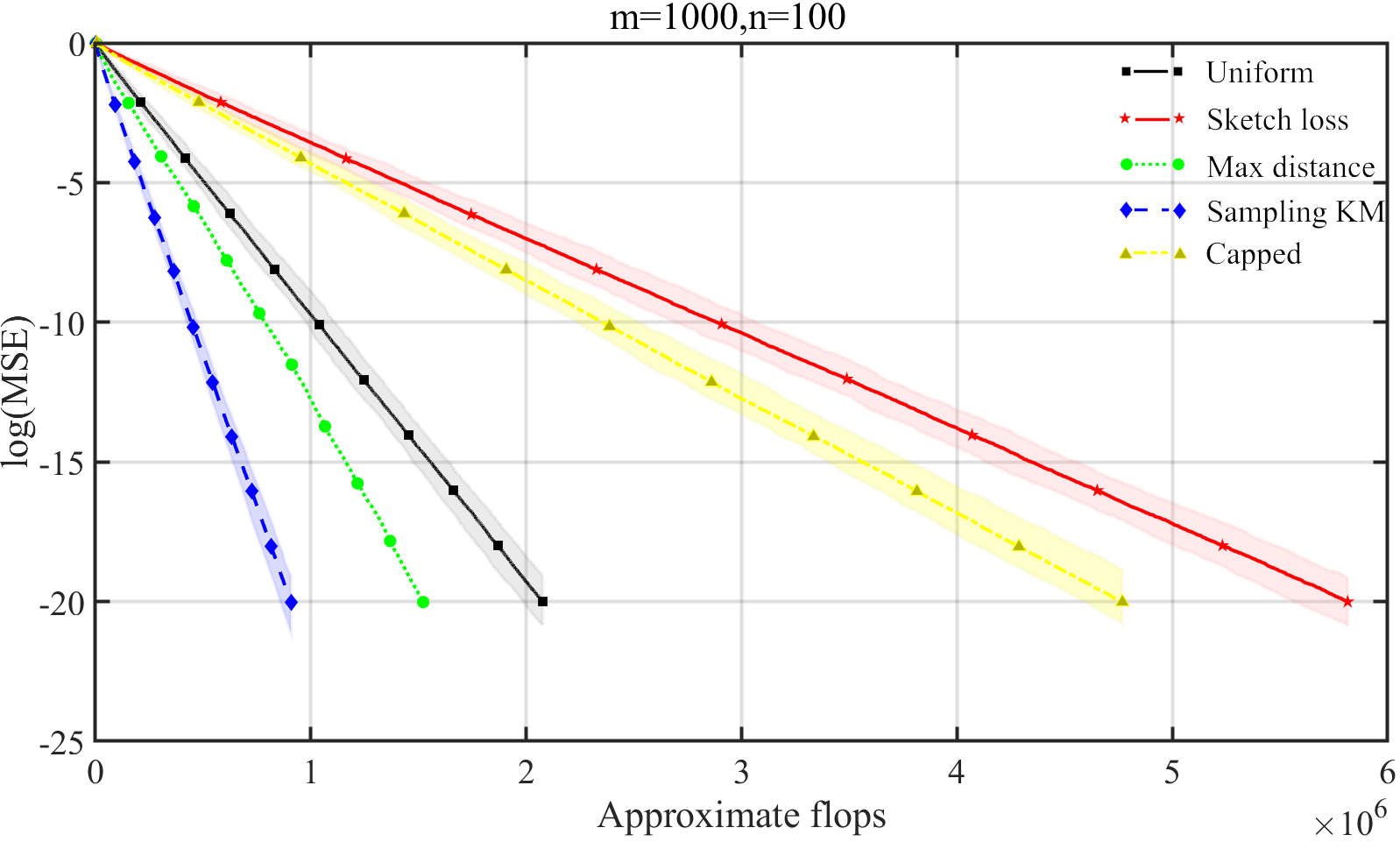}
	\caption{}
\end{subfigure}
\caption{A comparison between different sampling strategies for randomized Kaczmarz methods. MSE are averaged over 100 trials and are plotted against the iteration and corresponding approximate
	flops aggregated over the computations respectively. Subplots on the first row show convergence for underdetermined systems, and those on the second column show the convergence on an overdetermined systems. Thick line
	shows median over 100 trials, light area is between min and max, and darker area indicate 25th and 75th quantile.}
\label{figure01}
\end{figure}

In the tests, we are going to solve the linear systems with two kinds of different coefficient matrices $\mathbf{ A}\in\mathbb{R}^{m\times n}$. One is the random matrix, generated by using the Matlab function
'randn'. In our implementations, the solution $\overline{x}\in\mathbb{R}^n$
is randomly generated by using the Matlab function "randn" as well. The nonzero locations are chosen randomly according to the sparsity. The measurement $b\in\mathbb{R}^m$ is calculated by $A\hat{x}$. The other type of matrices is originated from Computed Tomography (CT).
The signal and matrix are generated by AIRtools toolbox in \cite{2017AIR}. All computations are started from the
initial vector $x^0 = 0$. The mean square error (MSE) is defined as
\begin{eqnarray}
\textrm{MSE} = \frac{\|x^k-\overline{x}\|^2}{\|\overline{x}\|^2},
\end{eqnarray}
All of the experiments are carried out by Matlab(Version R2019a) on a personal computer with 2.70GHZ CPU(Intel(R) Core(TM) i7-6820HQ), 32GB memory, and Microsoft Windows 10 operation system.

\subsection{MSE per iteration}

We first investigate the behaviours of MSE in terms of the iterations. We divide the testing matrices into two categories $m/n=1.5$(or $n/m=1.5$) and $m/n\approx3.3$(or $n/m\approx3.3$). For each category,  both  under-determined and over-determined situations are considered. The sparsity of the signal is $30$. The results are shown in Figure \ref{figure1} and Figure \ref{figure2}.
\begin{figure}
	\begin{subfigure}{0.5\textwidth}
		\includegraphics[width=1\textwidth]{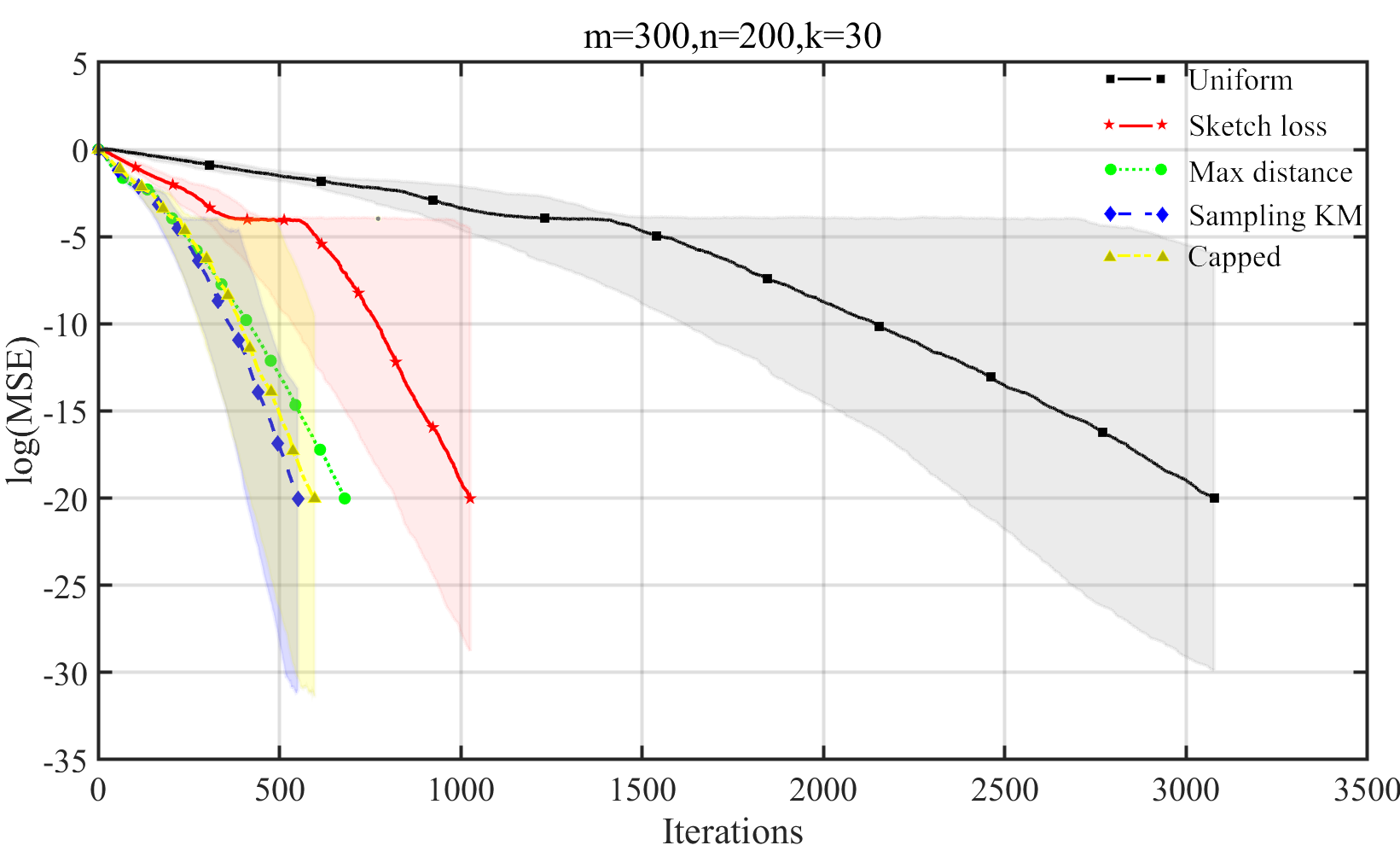}
		\caption{}
	\end{subfigure}
\begin{subfigure}{0.5\textwidth}
	\includegraphics[width=1\textwidth]{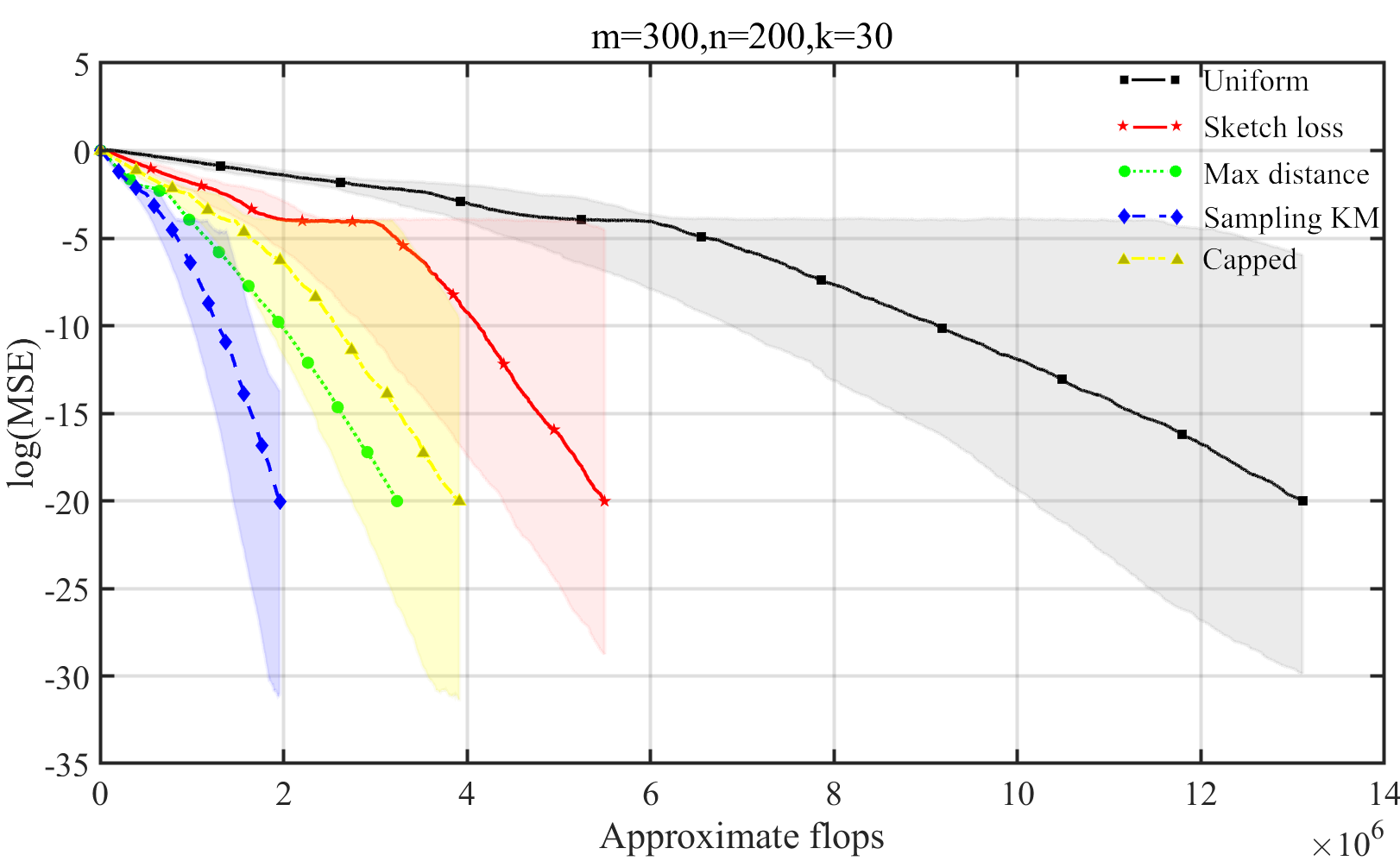}
	\caption{}
\end{subfigure}
	\begin{subfigure}{0.5\textwidth}
		\includegraphics[width=1\textwidth]{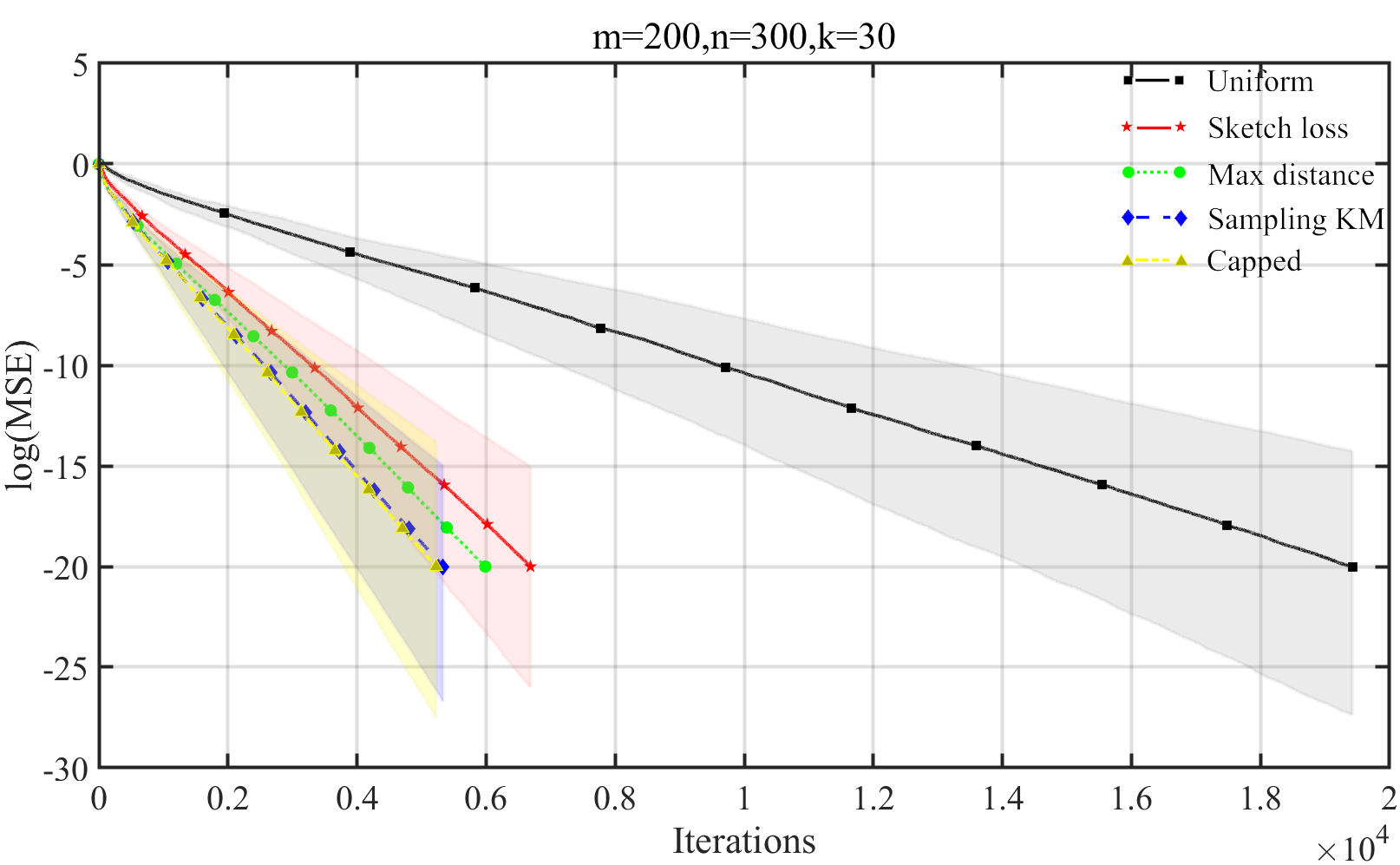}
		\caption{}
	\end{subfigure}
	\begin{subfigure}{0.5\textwidth}
		\includegraphics[width=1\textwidth]{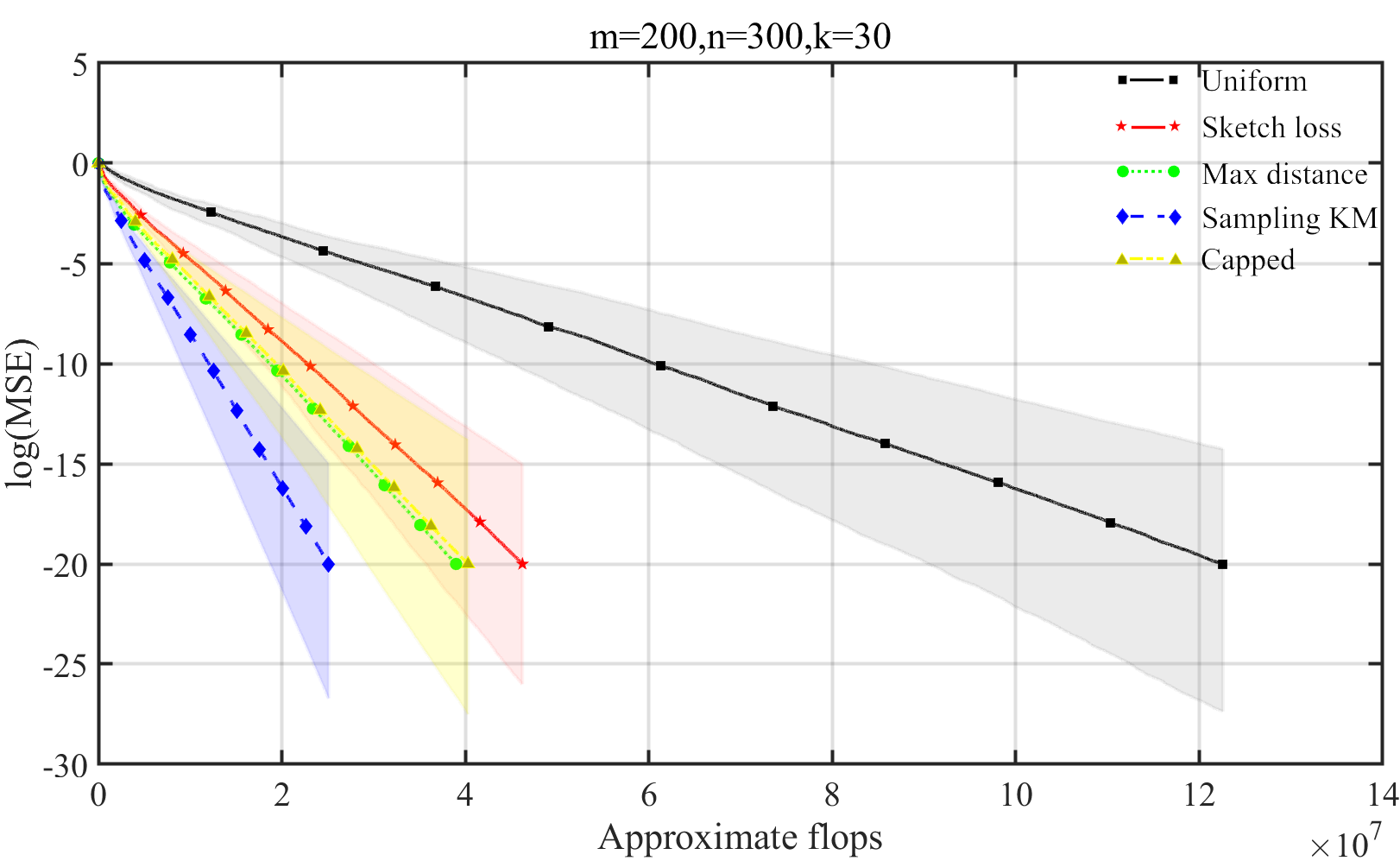}
		\caption{}
	\end{subfigure}
\caption{A comparison between different sampling strategies for sparse Kaczmarz methods. MSE are averaged over 100 trials and are plotted against the iteration and corresponding approximate
	flops aggregated over the computations respectively. Subplots on the first row show convergence for underdetermined systems, and those on the second column show the convergence on an overdetermined systems. Thick line
	shows median over 100 trials, light area is between min and max, darker area indicate 25th and 75th quantile.}
\label{figure1}
\end{figure}

\begin{figure}
		\begin{subfigure}{0.5\textwidth}
		\includegraphics[width=1\textwidth]{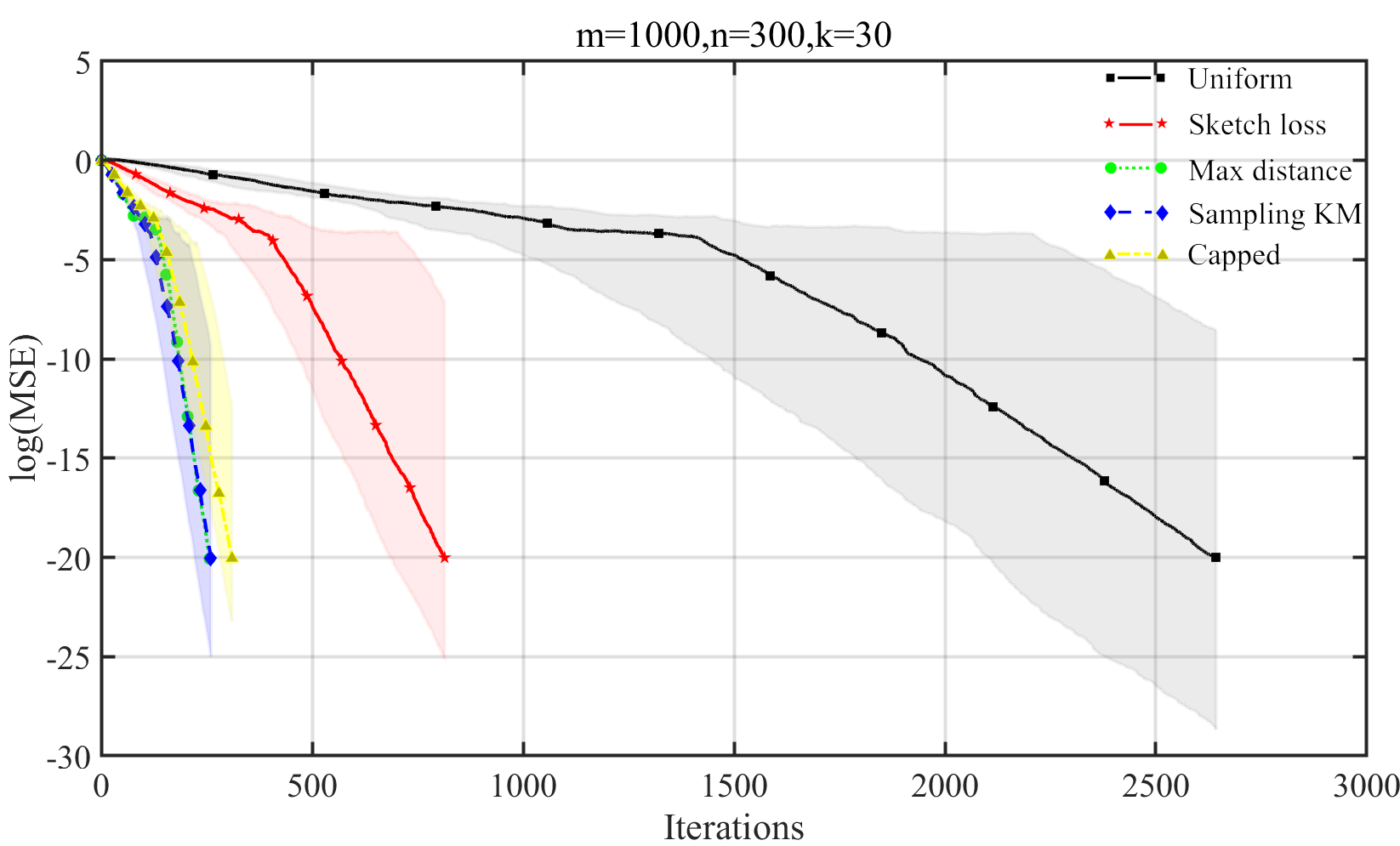}
		\caption{}
	\end{subfigure}
	\begin{subfigure}{0.5\textwidth}
		\includegraphics[width=1\textwidth]{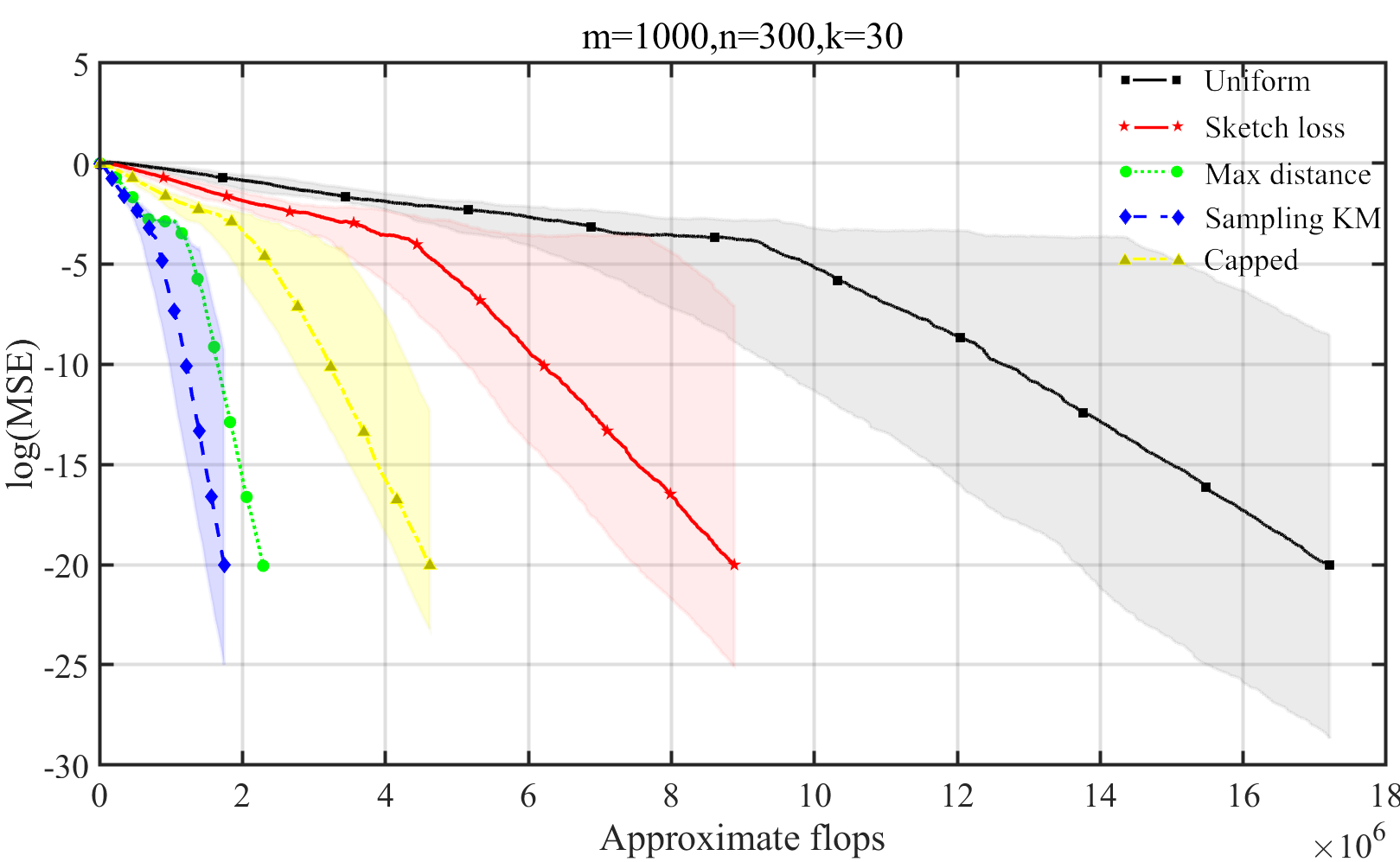}
		\caption{}
	\end{subfigure}
	\begin{subfigure}{0.5\textwidth}
		\includegraphics[width=1\textwidth]{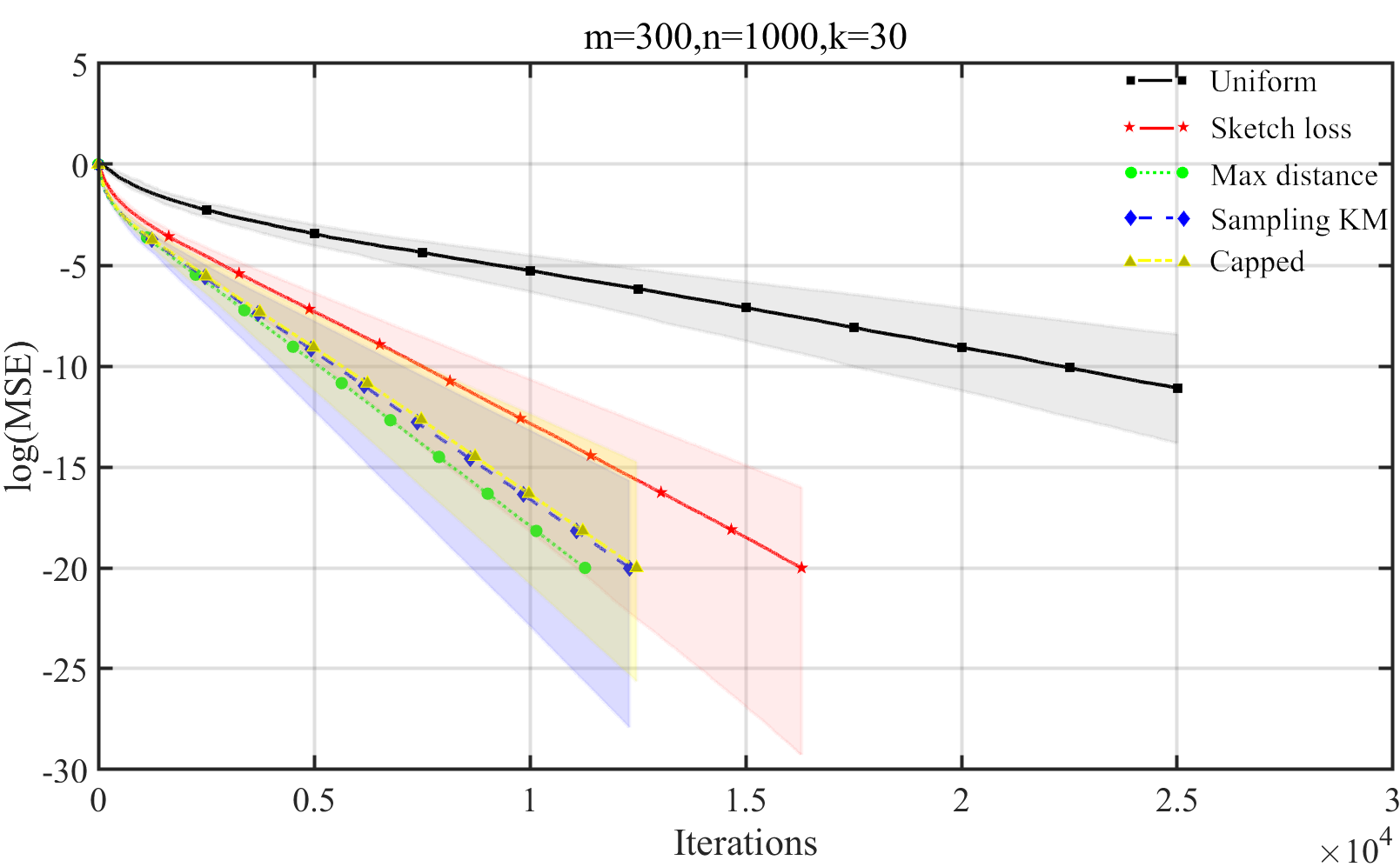}
		\caption{}
	\end{subfigure}
	\begin{subfigure}{0.5\textwidth}
		\includegraphics[width=1\textwidth]{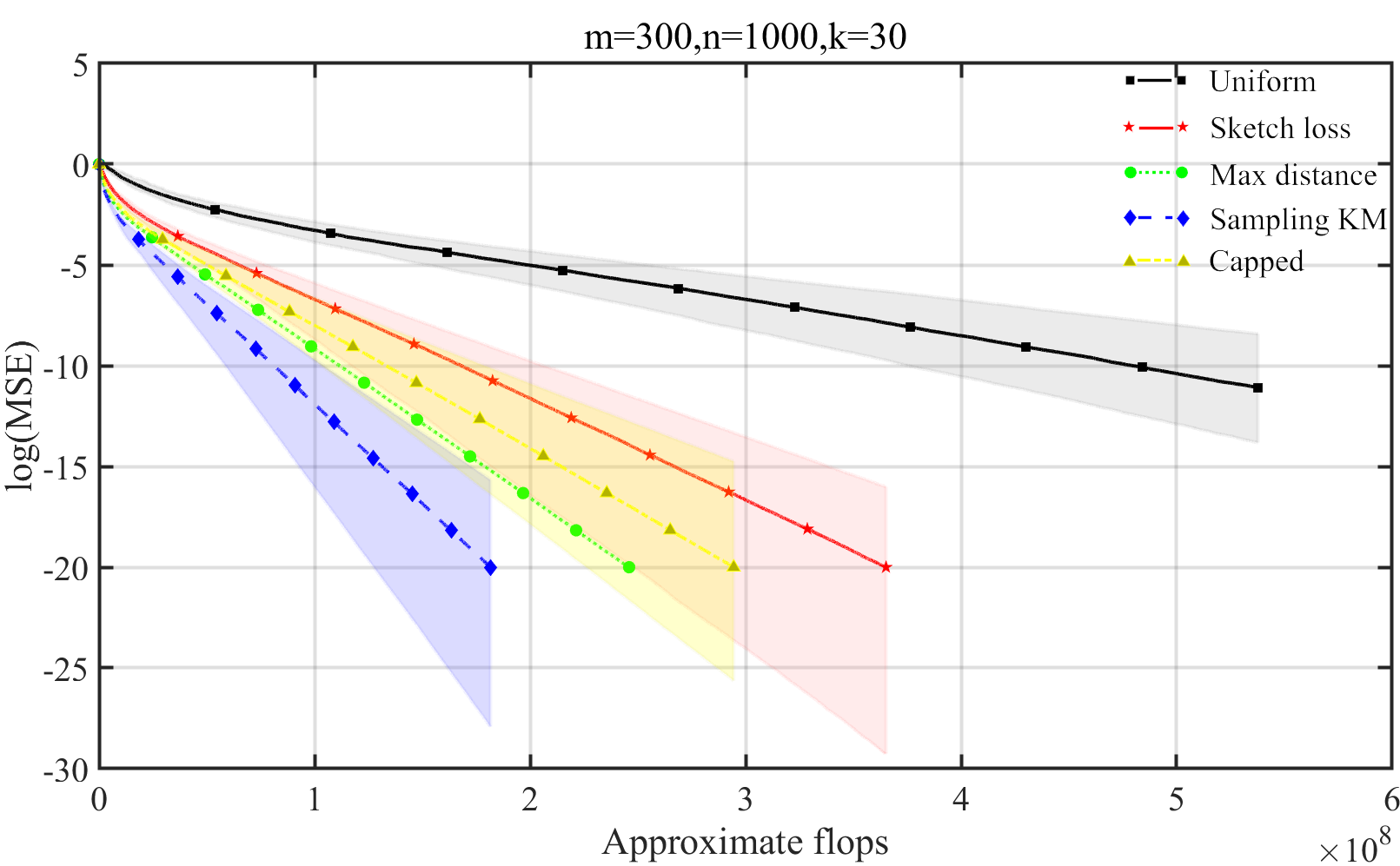}
		\caption{}
	\end{subfigure}
\caption{A comparison between different sampling strategies for sparse Kaczmarz methods. MSE are averaged over 100 trials and are plotted against the iteration and corresponding approximate
	flops aggregated over the computations respectively. Subplots on the first row show convergence for underdetermined systems, and those on the second column show the convergence on an overdetermined systems. Thick line
	shows median over 100 trials, light area is between min and max, darker area indicate 25th and 75th quantile.}
\label{figure2}
\end{figure}
For the over-determined situation, from the results in Figure \ref{figure1} and Figure \ref{figure2} we can find that the method equipped with the sampling Kaczmarz-Motzkin rule achieve a best performance,  followed by the max-distance sampling rule, the capped rule, proportional to the sketched loss rule, and the uniform rule. For the under-determined situation, the method using the sampling Kaczmarz-Motzkin rule also achieves the best performance, the max-distance sampling rule and capped sampling rule have a similar performance, next is the proportional to the sketched loss rule, and the last is the uniform sampling rule. The sparse Kaczmarz method using these sampling rules which has a larger convergence rate shown in section 5 usually has a faster convergence speed. But the method equipped with the sampling kaczmarz-Motzkin rule usually achieve the best performance.

\subsection{Computational cost per iteration}
In \cite{Gower2019}, it has completely analyzed the computational costs of each sampling rule. Combining with the cost $12n+nln(n)$ for solving the minimization problem \eqref{subo} and the cost $5n$ for soft thresholding in each iteration\cite{2016Linear}, we summarize the computational costs for the randomized sparse Kaczmarz method with different rules in Table \ref{costs}. Note that, for all the adaptive sampling rule methods, the method equipped with the capped sampling rule requires the most flops in each iteration.

In the experiments, we record the number of flops required for each sampling rule to make the MSE decrease, the results are displayed in Figure \ref{figure1} and Figure \ref{figure2}. We can find that although the adaptive methods demand more computational burdens than non-adaptive method in each iteration, these adaptive methods still need fewer operational flops to achieve a lower error. Specially, the sparse kaczmarz method using the sampling Kaczmarz-Motzkin rule demands the least computational cost among all adaptive sampling rules.
\begin{table}
	\caption{Summary of the computational costs for the sparse Kaczmarz method using different sampling strategies, where $m,n$ are numbers of the row and columns of the matrix,  $\beta_k$ is the sampling size.}\label{costs}
	\centering
	\begin{tabular}{|c|c|c|}
		\hline Sampling Strategy & Flops & Rate Bound Shown In\\
		\hline {Uniform} & {$21n+nln(n) $} & {Theorem $7$\cite{2016Linear}}\\
		
		\hline{$p_{i} \propto\left\|A_{i:}\right\|_{2}^{2}$} &{$21n+nln(n) $} &{Theorem 7\cite{2016Linear} }\\
		\hline {Max-distance} & {$m+17n+nln(n)$} & {Section 4.2.1} \\
		\hline{\text{Sketched loss}} & {$2m+17n+nln(n)$} &{Section 4.2.2}\\
		\hline {Capped} & {$5m+17n+nln(n)$}&{Section 4.2.3}\\
		\hline Sampling & \multirow{2}{*}{$\beta_k+17n+nln(n)$} & \multirow{2}{*}{Theorem 2\cite{Yuan2021}}\\
		Kaczmarz-Motzkin&&\\
		\hline
	\end{tabular}
\end{table}

\subsection{Phantom picture}
In this test, we will study an academic tomography problem to test the performance of sparse Kaczmarz method utilizing different sampling rules. As it is well known, the underlying model in this experiment consists of straight X-rays which penetrate the object, afterwards the damping is recorded. This physical process can be formulated as a linear equation  model $Ax=b$.
\begin{figure}
	\begin{subfigure}{0.5\textwidth}
		\includegraphics[width=1\textwidth]{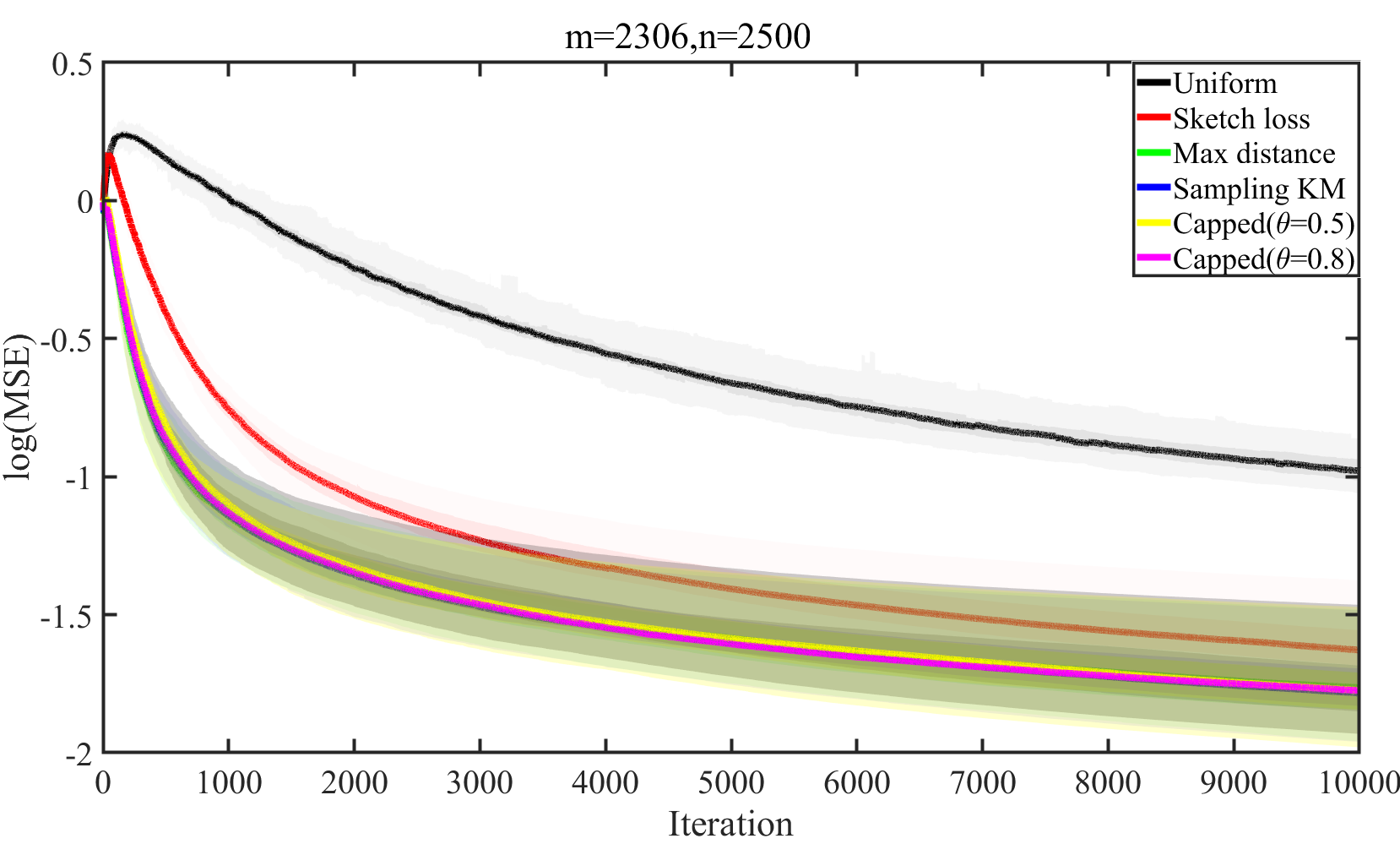}
		\caption{}
	\end{subfigure}
	\begin{subfigure}{0.5\textwidth}
		\includegraphics[width=1\textwidth]{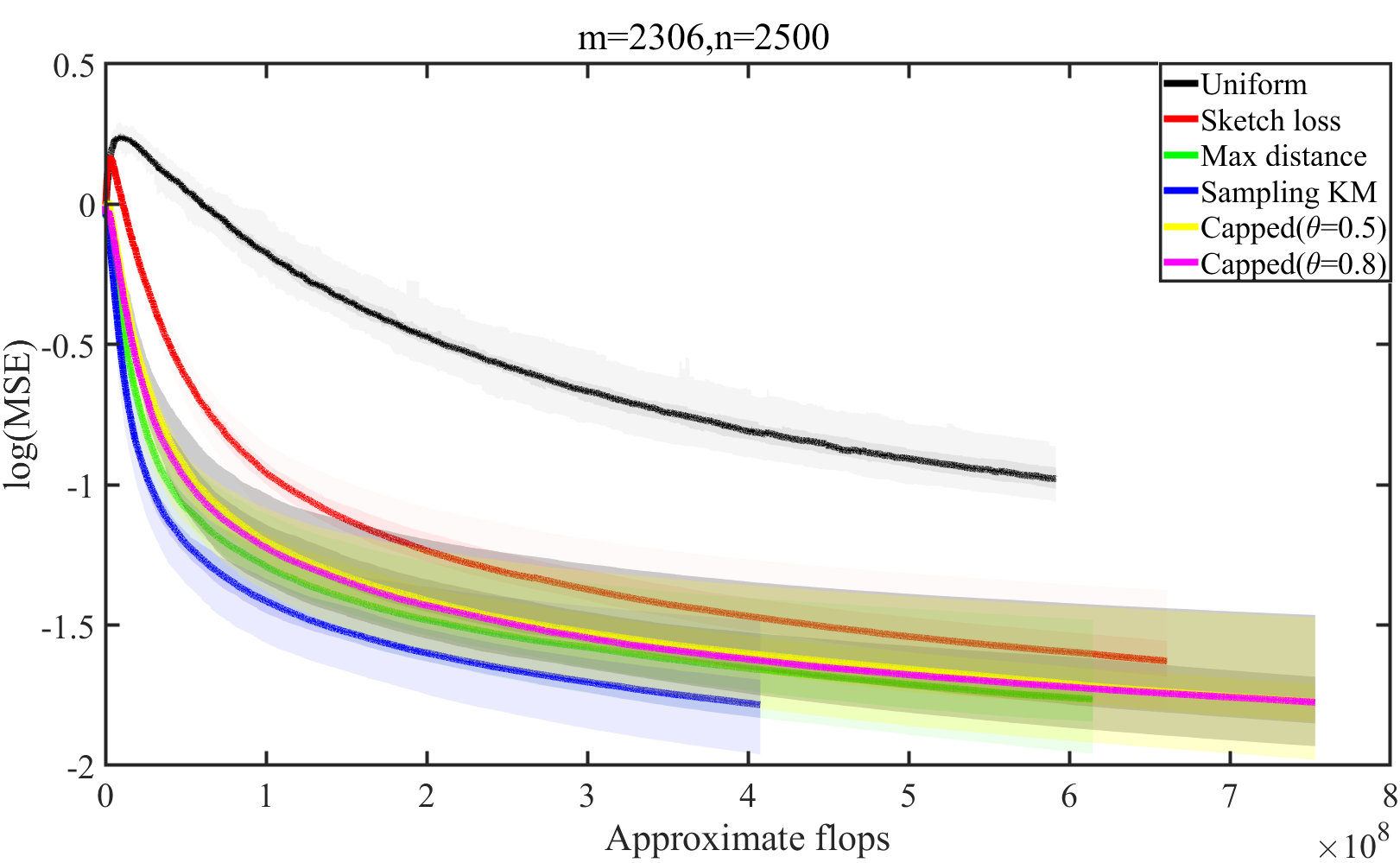}
		\caption{}
	\end{subfigure}
\caption{A comparison between different selection strategies for the randomized sparse Kaczmarz methods. Squared error norms were averaged over 100 trials and are plotted against the approximate
flops aggregated over the computations that occur at each iteration. Confidence intervals indicate the middle
95
right show the convergence on an over-determined systems.}
\label{figure3}
\end{figure}



AIRtools toolbox \cite{2017AIR} is utilized to generate the matrix $A$. In this test, $n=2500$ and $m=2049$. The image of interest is shepplogan shown in Figure \ref{imagedata}(a), which is sparse. Thus we can apply sparse kaczmarz method to recover it from dump. In this test, we not only compare the convergence rate and computational cost of sparse Kaczmarz method utilizing different sampling rules, but also the PSNR (Peak Signal to Noise Ratio) of the recovered images. The results are shown in Figure \ref{figure3} and Figure \ref{imagedata}.
\begin{figure}
		\begin{subfigure}{0.33\textwidth}
		\centering
		\includegraphics[width=1\textwidth]{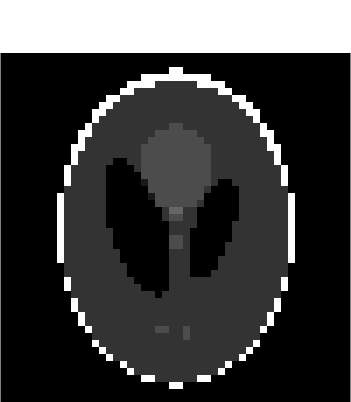}
		\caption{Ground Truth.}
	\end{subfigure}
	\begin{subfigure}{0.33\textwidth}
		\centering
		\includegraphics[width=1\textwidth]{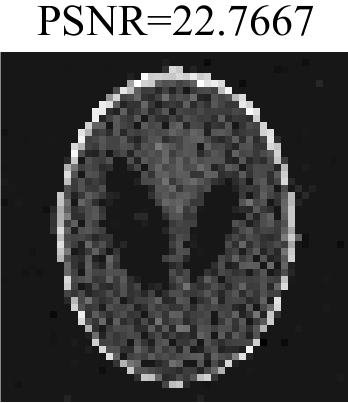}
		\caption{Uniform.}
	\end{subfigure}
	\begin{subfigure}{0.33\textwidth}
		\centering
		\includegraphics[width=1\textwidth]{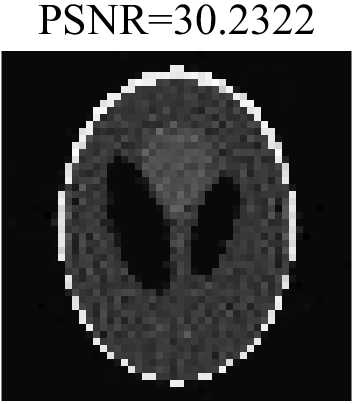}
		\caption{Sketched loss.}
	\end{subfigure}
	\begin{subfigure}{0.33\textwidth}
		\centering
		\includegraphics[width=1\textwidth]{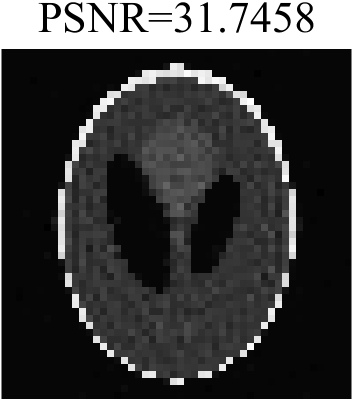}
		\caption{Max distance.}
	\end{subfigure}
		\begin{subfigure}{0.33\textwidth}
	\centering
	\includegraphics[width=1\textwidth]{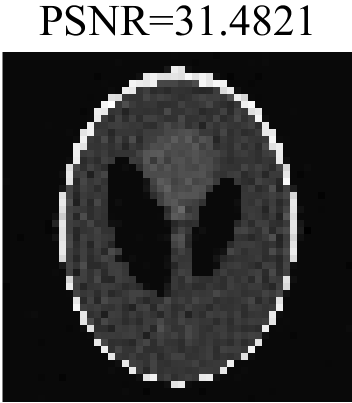}
	\caption{Capped $\theta=0.5$.}
\end{subfigure}
	\begin{subfigure}{0.33\textwidth}
		\centering
		\includegraphics[width=1\textwidth]{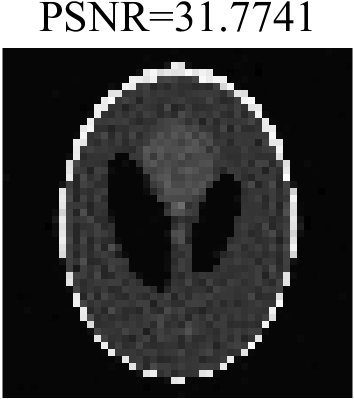}
		\caption{Sampling Kaczmarz-Motzkin.}
	\end{subfigure}
	\caption{Experimental results. (a) is the ground truth. (b) is the result recovered by the randomized sparse Kaczmarz method using the uniform sampling strategy. (c) is the result recovered by the randomized sparse Kaczmarz method using the sketched loss sampling strategy. (d) is the result recovered by the randomized sparse Kaczmarz method using the Max distance strategy. (e) is the result recovered by the randomized sparse Kaczmarz method using the sampling Kaczmarz-Motzkin strategy.(f)are the results recovered by the randomized sparse Kaczmarz method using the capped sampling strategy with parameters $\theta=0.5$.}\label{imagedata}
\end{figure}

From Figure \ref{imagedata}, we can find that the method using the sampling Kaczmarz-Motzkin rule has advantages over other sampling rules. It can recover the image with the largest PSNR in a given step. At the same time, from Figure \ref{figure3}, we can still find that the method equipped with the sampling Kaczmarz-Motzkin rule still achieve the lowest MSE with least computational costs. Although the matrix is of rank deficiency, by utilizing the sparsity structure of image, we can still find the ground truth. It will shed light on more applications to utilize prior information to recover the signal of interest with fewer measurements.

\section{Conclusion}
In this paper, we proposed the adaptively sketched Bregman projection method to solve linear systems, flexible in finding solutions with certain structures. As a mathematical framework, the proposed method, equipped with adaptive sampling rules, is general enough to cover the randomized sparse Kaczmarz method and the adaptive sketch-and-project method. Theoretically, we showed the linear convergence of SBP with detailed rates. Numerically, we reported experiment results to demonstrate that the (sparse) Kaczmarz method equipped with the sampling Kaczmarz-Motzkin rule demands the least computational cost to achieve the lowest error.

In the future, we would like to extend the proposed SBP method to deal with matrix equations. 
\section*{Acknowledgments}
The paper is granted by the National Natural Science Foundation of China:11971480, 61977065, and the Hunan Province excellent youngster Foundation:2020JJ3038. Education Department of Hunan:CX20190016. This work does not have any conflicts of interest.

\bibliographystyle{unsrt}
\small
\bibliography{1.bib}

\end{document}